\documentclass[11pt]{article}
\usepackage[margin=1in]{geometry}
\usepackage{amssymb,amsthm,amsmath}
\usepackage{mathtools}
\usepackage{tikz-cd}
\usetikzlibrary{arrows,decorations.pathmorphing,backgrounds,positioning,fit,petri,calc,shapes.misc,decorations.markings,intersections}
\usepackage[affil-it]{authblk}
\usepackage{enumerate}
\usepackage[all]{xy}
\usepackage{setspace}
\usepackage{hyperref}

\usepackage{float}
\usepackage{titling}
\newcommand{\subtitle}[1]{%
  \posttitle{%
    \par\end{center}
    \begin{center}\large#1\end{center}
    \vskip0.5em}%
}
\makeatletter
\newcommand{\oset}[3][0ex]{%
  \mathrel{\mathop{#3}\limits^{
    \vbox to#1{\kern-2\ex@
    \hbox{$\scriptstyle#2$}\vss}}}}
\makeatother

\newtheorem{theorem}{Theorem}[section]
\newtheorem{proposition}{Proposition}[section]
\newtheorem{lemma}{Lemma}[section]
\theoremstyle{definition}
\newtheorem{definition}{Definition}[section]

\providecommand{\norm}[1]{\lVert#1\rVert}
\providecommand{\abs}[1]{\lvert#1\rvert}

\title{Extensions of Bundles of C*-algebras}
\author{Jeremy Steeger\\\texttt{jsteeger@uw.edu}}
\affil{{Department of Philosophy}\\{University of Washington}}
\author{Benjamin H.~Feintzeig\\ \texttt{bfeintze@uw.edu}}
\affil{{Department of Philosophy}\\{University of Washington}}
\date{\today}

\begin{document}

\maketitle

\begin{abstract}
Bundles of C*-algebras can be used to represent limits of physical theories whose algebraic structure depends on the value of a parameter.  The primary example is the $\hbar\to 0$ limit of the C*-algebras of physical quantities in quantum theories, represented in the framework of strict deformation quantization.  In this paper, we understand such limiting procedures in terms of the extension of a bundle of C*-algebras to some limiting value of a parameter.  We prove existence and uniqueness results for such extensions.  Moreover, we show that such extensions are functorial for the C*-product, dynamical automorphisms, and the Lie bracket (in the $\hbar\to 0$ case) on the fiber C*-algebras.
\end{abstract}
\tableofcontents
\newpage

\section{Introduction}
\label{sec:intro}

In this paper, we are interested in limiting procedures in which the algebraic structure of a theory varies depending on the numerical value of some parameter.  In many cases, the physical quantities of such a theory can be represented by a C*-algebra of quantities, and the corresponding limit can be represented by a bundle of C*-algebras \cite{Di77,KiWa95}.  The primary example of this is the formulation of the classical $\hbar\to 0$ limit of quantum theories in the framework of strict deformation quantization.  The general theory is presented in \cite{Ri89,Ri93,La98b,La06,La17}, and many examples are investigated in \cite{La93b,La93a,La98a,BiHoRi04b,HoRi05,HoRiSc08,BiGa15,vN19}.  Under modest conditions, a strict deformation quantization determines a continuous bundle of C*-algebras, with so-called equivalent quantizations determining the same continuous bundle.  Beyond the $\hbar\to 0$ limit, analogous structures have been used to represent scaling limits for renormalization in quantum field theory by \cite{BuVe95,BuVe98}, who employ particular algebras of sections of a bundle of C*-algebras that they call scaling algebras.  These structures are investigated further in \cite{Bu96,Bu96a}, and \cite{BuVe95} also make the suggestion that these tools may apply to the non-relativistic $c\to \infty$ limit of relativistic theories.  Thus, bundles of C*-algebras have important applications to  limits of physical theories.

However, in many cases the definition of the relevant bundle, and so the representation of the limiting procedure, is presented by stipulating the structure of the C*-algebra at the limiting value of the parameter.  For example, strict deformation quantizations are typically presented by \emph{starting} with a classical Poisson algebra densely embedded in a commutative C*-algebra at $\hbar = 0$.  One would like to know whether this limiting C*-algebra is somehow determined by the C*-algebras that form the rest of the bundle for $\hbar >0$, and similarly for the other examples.  We pose the associated question as concerning the existence and uniqueness of extensions of bundles to enlarged parameter spaces.  For example, for the $\hbar\to 0$ limit, one wants to know whether a bundle over the base space $(0,1]$ uniquely extends to a bundle over $[0,1]$, or equivalently whether there is a unique algebra at $\hbar = 0$ that can be continuously glued to the existing bundle.  We answer this question in the affirmative.

Furthermore, as \cite{La02,La02a} has conjectured that quantization is functorial, one can pose an analogous question for the inverse procedure of the classical limit, or for extensions of bundles of C*-algebras more generally.  
We use our existence and uniqueness results here to establish a number of senses in which limiting procedures represented by extensions of bundles of C*-algebras are functorial.  This holds for a variety of physically significant structures on the fiber C*-algebras, including the C*-product, dynamical automorphisms, and the Lie bracket (in the $\hbar\to 0$ case).  However, we note that our results here differ from the investigation of \cite{La02,La02a} because we focus on classes of morphisms that are no larger than the class of *-homomorphisms (i.e., morphisms that preserve at least C*-algebraic structure), whereas Landsman treats Hilbert bimodules as morphisms, which form a class larger than the *-homomorphisms (i.e., they preserve somewhat less structure).  Still, we believe the functoriality of quantization and of the classical limit to be closely connected, and we believe that further work may bring these investigations together.

We also mention in passing that the physical and philosophical significance of the limiting construction and its functoriality will be addressed in a companion philosophical paper, where we argue that the representation of the $\hbar\to 0$ limit with bundles of C*-algebras provides a kind of intertheoretic reduction.

The current paper proceeds as follows.  In \S\ref{sec:quant}, we motivate our investigation in more detail by reviewing how a strict deformation quantization determines a continuous bundle of C*-algebras.  This section also takes the opportunity to define relevant notions, and the associated \ref{app:quanttobun} shows that the determination of a bundle from a strict quantization is functorial.  In \S\ref{sec:uniform}, we provide a definition of a novel kind of bundle of C*-algebras, called a \emph{$\mathit{UC}_b$ bundle}, which we will use to establish our results concerning extensions.  We prove in the associated \ref{app:equivdefs} that a category of $\mathit{UC}_b$ bundles is equivalent to a category of continuous bundles as standardly defined by \cite{KiWa95}.  In \S\ref{sec:ext}, we prove that our $\mathit{UC}_b$ bundles have unique extensions to base spaces with limiting values of parameters.  Finally, \S\ref{sec:func} establishes that our constructions of bundle extensions and limiting C*-algebras are functorial with respect to C*-algebraic structure (\S\ref{sec:C*}), dynamical structure (\S\ref{sec:dyn}), and Lie bracket structure (\S\ref{sec:lie}).

\section{Motivation from Quantization}
\label{sec:quant}
We represent the physical quantities of a quantum system with elements of a C*-algebra.  A C*-algebra $\mathfrak{A}$ is an associative, involutive, complete normed algebra satisfying the C*-identity: $\norm{A^*A} = \norm{A}^2$ for all $A\in\mathfrak{A}$.  The canonical examples of C*-algebras are commutative algebras of bounded functions on a locally compact topological space and possibly non-commutative algebras of bounded operators on a Hilbert space.  We employ commutative algebras of functions on a classical phase space in classical physics and non-commutative algebras of operators satisfying a version of the canonical (anti-) commutation relations in quantum theories.  Thus, using C*-algebras provides a unified framework for investigating the relationship between classical and quantum theories.  For mathematical background, we refer the reader to \cite{Sa71,Di77,KaRi97}.  See \cite{BrRo87,BrRo96,Ha92} for the C*-algebraic approach to quantum physics.  With this background in place, we can use C*-algebras to analyze the classical limit.

A strict quantization provides extra structure to ``glue together'' a family of C*-algebras indexed by the parameter $\hbar$, as follows.

\begin{definition}[\cite{La98b}]
A \emph{strict quantization} of a Poisson algebra $(\mathcal{P},\{\cdot,\cdot\})$ consists in a locally compact topological space $I\subseteq \mathbb{R}$ containing $0$, a family of C*-algebras $(\mathfrak{A}_\hbar)_{\hbar\in I}$ and a family of linear \emph{quantization maps} $(\mathcal{Q}_\hbar: \mathcal{P}\to\mathfrak{A}_\hbar)_{\hbar\in I}$.  We require that $\mathcal{P}\subseteq\mathfrak{A}_0$, $\mathcal{Q}_0$ is the inclusion map, and for each $\hbar\in I$, $\mathcal{Q}_\hbar[\mathcal{P}]$ is norm dense in $\mathfrak{A}_\hbar$.  Further, we require that the following conditions hold for all $A,B\in \mathcal{P}$:
\begin{enumerate}[(i)]
    \item \emph{Von Neumann's condition.} $\lim_{\hbar\to 0}\norm{\mathcal{Q}_\hbar(A)\mathcal{Q}_\hbar(B) - \mathcal{Q}_\hbar(AB)}_\hbar = 0$;
    \item \emph{Dirac's condition.} $\lim_{\hbar\to 0}\norm{\frac{i}{\hbar}[\mathcal{Q}_\hbar(A),\mathcal{Q}_\hbar(B)] - \mathcal{Q}_\hbar(\{A,B\})}_\hbar = 0$;
    \item \emph{Rieffel's condition.} the map $\hbar\mapsto \norm{\mathcal{Q}_\hbar(A)}_\hbar$ is continuous.
    \end{enumerate}
    
\noindent A strict \emph{deformation} quantization is a strict quantization that satisfies the additional requirement that for each $\hbar\in I$, $\mathcal{Q}_\hbar [\mathcal{P}]$ is closed under multiplication and is nondegenerate, i.e., $\mathcal{Q}_\hbar (A) = 0$ if and only if $A=0.$
\end{definition}

\noindent \cite{Ri89,Ri93} uses this approach to define a ``deformed'' product on the classical algebra, which he uses in turn to specify the quantization conditions. We work in the opposite direction, defining the elements of the classical algebra using information away from $\hbar=0$.  We will mostly ignore the deformation condition, but still refer to strict deformation quantizations to distinguish them from other approaches like geometric quantization.

In a strict deformation quantization, one can understand the classical limit of a quantity $\mathcal{Q}_\hbar(A)$ in $\mathfrak{A}_\hbar$ in the quantum theory to be the classical quantity $A\in \mathcal{P}$.  Similarly, we can take classical limits of states by defining a \emph{continuous field of states} as a family of states $\omega_\hbar\in\mathcal{S}(\mathfrak{A}_\hbar)$ for each $\hbar\in I$ such that the map $\hbar\mapsto \omega_\hbar(\mathcal{Q}_\hbar(A))$ is continuous for each $A\in \mathcal{P}$.  In this case, the classical limit of such a continuous field of states is understood to be the classical state $\omega_0\in\mathcal{S}(\mathfrak{A}_0)$.  Thus, a strict deformation quantization provides enough structure to represent the classical limits of states and quantities in quantum theories. In fact, there is a close connection between the satisfaction of Rieffel's condition in a strict quantization and the existence of a continuous field of states converging to each classical state; see the proof of Theorem 4 in \cite[p. 33]{La93b} for more details. However, in the present paper, we will mostly ignore classical limits of states and restrict our focus kinematical quantities and their associated dynamics.  We leave a more thorough treatment of states for future work.

In general, there may be different strict deformation quantizations of the same Poisson manifold. If two strict quantizations ${\mathcal{Q}}_\hbar$ and  ${\mathcal{Q}}_\hbar'$ of a given Poisson algebra $\mathcal{P}$ employ the same family of C*-algebras, but possibly differ in their quantization maps, then the quantizations are called \emph{equivalent} just in case
\begin{align}\label{eq:quant_equiv}
    \lim_{\hbar\to 0}\left\Vert{\mathcal{Q}}_\hbar(A) - {\mathcal{Q}}_\hbar ' (A) \right\Vert_\hbar = 0
\end{align}
for all $A\in \mathcal{P}$.  Equivalent quantizations share the same behavior in the limit as $\hbar\to 0$.  One can encode this common behavior in a further object, variously called a continuous bundle or field of C*-algebras, which itself has enough structure to understand the classical limits of quantities and states.  We consider an existing definition in the literature before providing an alternative{\textemdash}and in a sense, \emph{equivalent}{\textemdash}definition that will be more useful for our purposes.

\begin{definition}[\cite{KiWa95}]
\label{def:vbun}
A \emph{$C_0$ bundle of C*-algebras} over a topological space $I$ is a family of C*-algebras $(\mathfrak{A}_\hbar)_{\hbar\in I}$, a C*-algebra $\mathfrak{A}^0$ called the collection of \emph{vanishingly continuous sections}, and a family of *-homomorphisms $(\phi_\hbar: \mathfrak{A}^0\to \mathfrak{A}_\hbar)_{\hbar\in I}$ called \emph{evaluation maps}, which we require to satisfy the following conditions:
\begin{enumerate}[(i)]
\item \emph{Fullness.} Each evaluation map $\phi_\hbar$ is surjective and the norm of each $a\in \mathfrak{A}^0$ is given by $\norm{a} = \sup_{\hbar\in I}\norm{\phi_\hbar(a)}_\hbar$.
\item \emph{Vanishing completeness.} For each continuous function vanishing at infinity $f\in C_0(I)$ and $a\in\mathfrak{A}^0$, there is an element $fa\in\mathfrak{A}^0$ such that $\phi_\hbar(fa) = f(\hbar)\phi_\hbar(a)$.
\item \emph{Vanishing continuity.} For each $a\in\mathfrak{A}^0$, the function $N_a:\hbar\mapsto \norm{\phi_\hbar(a)}_\hbar$ is in $C_0(I)$.
\end{enumerate}
\end{definition}

\noindent Other authors typically call these structures simply ``continuous bundles of C*-algebras.''  We instead use the modifier ``$C_0$" to emphasize the use of $C_0(I)$, the continuous functions vanishing at infinity. This will help to distinguish these structures from the ones we define next. We will often think of the C*-algebras $\mathfrak{A}_\hbar$ as fibers above the values $\hbar\in I$, hence forming a bundle structure over the base space $I$.  A $C_0$ bundle of C*-algebras determines the continuity structure of the bundle by specifying the collection $\mathfrak{A}^0$ of vanishingly continuous sections through the bundle.

By a theorem of Landsman \cite[Theorem 1.2.4, p. 111]{La98b}, given a strict quantization, there is a unique $C_0$ bundle of C*-algebras containing among its sections the curves traced out by the quantization maps as $\hbar$ varies. More formally: given a strict quantization $((\mathfrak{A}_\hbar,\mathcal{Q}_\hbar)_{\hbar\in I}, \mathcal{P})$, under modest conditions there is a unique $C_0$ bundle of C*-algebras $((\mathfrak{A}_\hbar,\phi_\hbar)_{\hbar\in I},\mathfrak{A}^0)$ such that for each $A\in \mathcal{P}$, there is a continuous section $a\in\mathfrak{A}^0$ with $\phi_\hbar(a) = \mathcal{Q}_\hbar(A)$ for all $\hbar\in I$. We can thus speak of the $C_0$ bundle generated by a strict quantization. Moreover, Landsman's theorem shows equivalent quantizations generate the same bundle.  In this sense, the bundles encode an invariant structure among different quantization maps capturing the same behavior in the $\hbar\to 0$ limit.  We analyze in further detail the construction of a $C_0$ bundle of C*-algebras from a strict deformation quantization in \ref{app:quanttobun}, where we show this construction is functorial.

\section{$\mathit{UC}_b$ Bundles}
\label{sec:uniform}

The association of a C*-algebra of vanishingly continuous sections in this definition allows one to use many familiar tools to analyze such bundles.  Still, there are some drawbacks to $C_0$ bundles for our purposes.  In the next section, we will formulate our central question about the determination of the classical limit, or limits of algebraic physical theories more generally, as follows.  Suppose one knows the quantum kinematics for $\hbar>0$ but has no knowledge of the corresponding classical kinematics, i.e., suppose one has a bundle of C*-algebras $((\mathfrak{A}_\hbar,\phi_\hbar)_{\hbar\in (0,1]},\mathfrak{A}^0)$ over the base space consisting only of parameter values $\hbar>0$.  Under what conditions is there a unique algebra $\mathfrak{A}_0$ that, when appropriately glued to the given bundle of algebras, provides an extended continuous bundle?

One would like to construct the sections of an extended bundle over $[0,1]$ by continuously extending sections of the restricted bundle over $(0,1]$ to the point $\hbar = 0$ in the base space. But if one begins with a $C_0$ bundle of C*-algebras over $(0,1]$, then the condition of vanishing continuity implies that all of the vanishingly continuous sections necessarily tend, as $\hbar\to 0$, toward the $0$ element of any C*-algebra one tries to glue on. Thus, this strategy can only lead to trivial limits. In other words, one cannot, in this way, directly recover non-trivial information about the corresponding classical algebra at $\hbar = 0$.

Notice, however, that we can make a slight change to the algebra of sections in our definition of bundles to deal with this issue. Consider a simplified toy example. Although every function in $C_0((0,1])$ will have the same (trivial) limit as $\hbar\to 0$, other algebras avoid this pathology. One clear candidate is the familiar algebra of bounded, continuous functions on the interval, as functions in $C_b ((0,1])$ can have non-trivial limits at $\hbar = 0$. But functions in this algebra might also fail to have limits altogether, as demonstrated by the topologist's sine curve. To uniquely extend such a function, one might try to appeal to the embedding $i$ of $(0,1]$ in its Stone-\v{C}ech compactification $\beta (0,1] $. The topologist's sine $t$ can be viewed as a function from $(0,1]$ to the compact space $[0,1]$, and by the universal property of $i$, $t$ extends uniquely to a continuous map $\beta t$ from $\beta (0,1] $ to $[0,1]$. But it is at least unclear that this extension is one of physical interest.

A more natural candidate is the collection $\mathit{UC}_b((0,1])$ of uniformly continuous and bounded functions, which is also a C*-algebra \cite[Lemma 3.10, p. 77]{AlBo99}. In fact, it is a C*-subalgebra of $C_b ((0,1])$. Moreover, each function $f\in \mathit{UC}_b((0,1])$ has a unique, uniformly continuous extension to $[0,1]$, and the value at the limit may be non-zero. More generally, for arbitrary metric spaces $(I,d)$, each $f\in \mathit{UC}_b(I)$ can be uniquely extended to the completion of $I$ \cite[Lemma 3.11, p. 77]{AlBo99}. These considerations suggest constructing bundles of C*-algebras whose sections are uniformly continuous and bounded.

As an aside, it would be interesting to know more about the relationship between $\mathit{UC}_b(I)$ and the more familiar commutative C*-algebras $C_0(I)$ and $C_b(I)$. For example, we lack a simple characterization of the first algebra's spectrum. Based on \cite{LacruzLlavona1997}, we conjecture that the spectrum of $\mathit{UC}_b(I)$ is some quotient of that of $C_b(I)$---that is, some quotient of the Stone-\v{C}ech compactification of $I$. We leave a more detailed investigation of this question for future work (and we thank an anonymous referee for raising it).

In sum, we use the algebra $\mathit{UC}_b(I)$ in an attempt to preserve the virtue of having a C*-algebra of sections while also gaining some control over non-trivial limits. We propose the following definition for the task.

\begin{definition}
\label{def:ubun}
A \textit{$\mathit{UC}_b$ bundle of C*-algebras} over a metric space $(I,d)$ is a family of C*-algebras $(\mathfrak{A}_\hbar)_{\hbar\in I}$, a C*-algebra $\mathfrak{A}$ called the collection of \textit{uniformly continuous sections}, and a family of *-homomorphisms $(\phi_\hbar: \mathfrak{A}\to\mathfrak{A}_\hbar)_{\hbar\in I}$ called \textit{evaluation maps}, which we require to satisfy the following conditions:
\begin{enumerate}[(i)]
\item \emph{Fullness.} Each evaluation map $\phi_\hbar$ is surjective and the norm of each $a\in \mathfrak{A}$ is given by $\norm{a} = \sup_{\hbar\in I}\norm{\phi_\hbar(a)}_\hbar$.
\item \emph{Uniform completeness.} For each $f\in \mathit{UC}_b(I)$ and $a\in\mathfrak{A}$, there is an element $fa\in\mathfrak{A}$ such that $\phi_\hbar(fa) = f(\hbar)\phi_\hbar(a)$.
\item \emph{Uniform continuity.} For each $a\in\mathfrak{A}$, the function $N_a: \hbar\mapsto \norm{\phi_\hbar(a)}_\hbar$ is in $\mathit{UC}_b(I)$.
\end{enumerate}
\end{definition}
\noindent Notice that the only difference between $\mathit{UC}_b$ and $C_0$ bundles of C*-algebras is the use of $C_0(I)$ or $UC_b(I)$ in conditions (ii) and (iii).  In general, we will restrict our attention to both $\mathit{UC}_b$ and $C_0$ bundles of C*-algebras whose base space is a locally compact metric space.

We now formulate two relevant categories of bundles of C*-algebras that we will use in what follows.  To that end, we define a notion of morphism, or structure-preserving map, between bundles.  Recall that a metric map $\alpha:I\to J$ is one satisfying $d_J(\alpha(x),\alpha(y)) \leq d_I(x,y)$ for all $x,y\in I$, where $d_I$ and $d_J$ are the metrics on $I$ and $J$, respectively.

\begin{definition}
A \emph{homomorphism $\sigma: \mathcal{A}_I\to\mathcal{B}_J$ between (vanishingly or uniformly) continuous bundles of C*-algebras} $\mathcal{A}_I \!= \!((\mathfrak{A}_\hbar,\phi^I_\hbar)_{\hbar\in I},\mathfrak{A})$ and $\mathcal{B}_J \! =\! ((\mathfrak{B}_\hbar,\phi^J_\hbar)_{\hbar\in J},\mathfrak{B})$ is a pair of maps
\begin{equation}\label{eq:bundle_homomorphism}
    \sigma = (\alpha,\beta),
\end{equation}
where $\alpha: I\to J$ is a metric map, $\beta: \mathfrak{A}\to\mathfrak{B}$ is a *-homomorphism.  We further require the following condition of compatibility between $\alpha$ and $\beta$: for all $a_1,a_2\in \mathfrak{A}$ and $\hbar\in I$
\[\text{if $\phi^I_\hbar(a_1) = \phi^I_\hbar(a_2)$, then $\phi^J_{\alpha(\hbar)}(\beta(a_1)) = \phi^J_{\alpha(\hbar)}(\beta(a_2))$}.\]
A homomorphism $\sigma = (\alpha,\beta)$ is an \emph{isomorphism} if $\alpha$ is an isometry and $\beta$ is a *-isomorphism.
\end{definition}
\noindent The condition of compatibility between $\alpha$ and $\beta$ ensures that $\beta$ preserves fibers in the sense that it defines a *-homorphism from the fiber $\mathfrak{A}_\hbar$ to the fiber $\mathfrak{B}_{\alpha(\hbar)}$ (See Lemma \ref{lem:fibmorph} in \S\ref{sec:ext}).

We now define two categories of bundles of C*-algebras.  Throughout, we will restrict attention to the case where the base space $I$ of any of the structures considered is a locally compact metric space.

\begin{definition}
The category $\mathbf{VBunC^*Alg}$ consists in:
\begin{itemize}
    \item \textit{objects}: $C_0$ bundles of C*-algebras whose base space is a locally compact metric space,
    \[\mathcal{A}^V_I = \left( \left(\mathfrak{A}_\hbar,\phi^I_\hbar\right)_{\hbar\in I},\mathfrak{A}^0\right);\]
    \item \textit{morphisms}: bundle homomorphisms.
\end{itemize}
\end{definition}

\begin{definition}
The category $\mathbf{UBunC^*Alg}$ consists in:
\begin{itemize}
    \item \textit{objects}: $\mathit{UC}_b$ bundles of C*-algebras whose base space is a locally compact metric space,
    \[\mathcal{A}^U_I = \left ( \left(\mathfrak{A}_\hbar,\phi^I_\hbar\right)_{\hbar\in I},\mathfrak{A} \right);\]
    \item \textit{morphisms}: bundle homomorphisms.
\end{itemize}
\end{definition}
The category $\mathbf{VBunC^*Alg}$ encodes the structures picked out by the definition of $C_0$ bundles of C*-algebras.  On the other hand, the category $\mathbf{UBunC^*Alg}$ encodes the structures picked out by the definition of $\mathit{UC}_b$ bundles of C*-algebras.  There is a sense in which our proposed definition of $\mathit{UC}_b$ bundles is equivalent to the previous definition of $C_0$ bundles, i.e., one can construct a unique $\mathit{UC}_b$ bundle from each $C_0$ bundle, and vice versa.  In \ref{app:equivdefs}, we make this precise by establishing a categorical equivalence between $\mathbf{VBunC^*Alg}$ and $\mathbf{UBunC^*Alg}$.  We take this equivalence to justify our use of $\mathit{UC}_b$ bundles of C*-algebras in what follows.

Next, we will use $\mathit{UC}_b$ bundles of C*-algebras to analyze the existence and uniqueness of limits of families of C*-algebras.  

\section{Existence and Uniqueness of Extensions}
\label{sec:ext}

We want to consider the case where we begin with only information about a quantum theory at $\hbar>0$ and no information about the corresponding classical limit at $\hbar = 0$, or the analogous situation for limits of other parameters.

In other words, suppose we are given only the restriction of our $\mathit{UC}_b$ bundle of C*-algebras to a bundle over the base space given by the half open interval $(0,1]$.  In general, one can define the \emph{canonical restriction} of a bundle $\mathcal{A}_J = ((\mathfrak{A}_\hbar,\phi^J_\hbar)_{\hbar\in J}, \mathfrak{A})$ over a base space given by a metric space $J$.  If $I$ is further a locally compact metric space and $\alpha: I\to J$ is an isometric embedding, we will say that the canonical restriction of $\mathcal{A}_J$ from $J$ to $I$ along $\alpha$ is given by
\begin{gather}\label{eq:canonical_restriction}
    \begin{aligned}
        \mathcal{A}_{J|I} &:= \left (\left(\mathfrak{A}_\hbar,\phi^I_\hbar \right)_{\hbar\in I},\mathfrak{A}_{|I} \right ), \quad \text{where}\\ 
        &\mathfrak{A}_{|I} := \left \{\gamma\in\prod_{\hbar\in I}\mathfrak{A}_\hbar\ \middle| \ \text{for some $a\in\mathfrak{A}$, } \gamma(\hbar) = \phi_{\alpha(\hbar)}^J(a)\text{ for all $\hbar\in I$}\right\},\quad\text{and} \\
        &\phi^I_\hbar(\gamma) := \gamma(\hbar) \quad \text{ for all $\gamma\in\mathfrak{A}_{|I}$ and $\hbar\in I$}.
\end{aligned}
\end{gather}
This restriction removes the fibers outside of $\alpha[I]$ and truncates continuous sections from $J$ to $\alpha[I]\subseteq J$. One can check that this is indeed a $\mathit{UC}_b$ bundle in its own right. In particular, for $I = (0,1]$ and $J=[0,1]$, one can use the natural inclusion map $\alpha: (0,1]\to [0,1]$ to define the restricted bundle $\mathcal{A}_{J|I}$ resulting from a strict deformation quantization as above. Such a restricted bundle represents only the information in the quantum theory for $\hbar>0$ without the corresponding classical theory at $\hbar = 0$.

Given such a restriction, our questions are: can one reconstruct the C*-algebra of classical quantities $\mathfrak{A}_0$ from this restricted continuous bundle? And can one continuously glue $\mathfrak{A}_0$ to the restricted bundle in a way that recovers the original information about the $\hbar\to 0$ limit? We answer both questions in the affirmative. The result is a two-step procedure for constructing the C*-algebra at $\hbar = 0$, which we call ``extension-and-restriction.'' Starting with a bundle containing only information for $\hbar > 0$, we (uniquely) extend the bundle to one containing information at the accumulation point $\hbar = 0$; then we (uniquely) restrict this new bundle to the fiber algebra $\hbar = 0$ to exactly recover the classical theory.

We will prove our results in full generality for the case where our base space $I$ is an arbitrary locally compact metric space.  Our general result then applies immediately to the case where the base space is either $I = (0,1]$ or $I = \{1/N\ | N\in\mathbb{N}\}$, which are the most typical base spaces used in analyzing limits of quantum theories.  As such, we define a general notion of extension.
\begin{definition}
\label{def:ext}
Let $\mathcal{A}_I = ((\mathfrak{A}_\hbar,\phi^I_\hbar)_{\hbar\in I},\mathfrak{A})$ and $\mathcal{B}_J = ((\mathfrak{B}_\hbar,\phi^J_\hbar)_{\hbar\in J},\mathfrak{B})$ be $\mathit{UC}_b$ bundles of C*-algebras over locally compact metric spaces $I$ and $J$, respectively.
\begin{itemize}
    \item $\mathcal{B}_J$ is an \textit{extension} of $\mathcal{A}_I$ if there is a monomorphism of continuous bundles of C*-algebras $\sigma= (\alpha,\beta): \mathcal{A}_I\to \mathcal{B}_J$, i.e., a homomorphism where $\alpha$ and $\beta$ are both injective.
    \item $\mathcal{B}_J$ is a \textit{minimal extension} of $\mathcal{A}_I$ if, moreover, $\alpha$ and $\beta$ are both dense embeddings.
\end{itemize}
In either case, we say that the extension $\mathcal{B}_J$ is associated with $\alpha$ via $\sigma$.
\end{definition}

\noindent We use the term ``dense embedding'' above to mean something slightly different for $\alpha$ and for $\beta$. For $\alpha$ to be a dense embedding, it must be an isometric isomorphism between its domain and its image (equipped with the subspace topology) and its image must be dense in $J$ (equipped with the metric topology). Sometimes we will call $\alpha$ a ``dense, isometric embedding" for emphasis; although, note that we do not require $\alpha$ to be bijective.  For $\beta$ to be a dense embedding, nothing more is required other than that its image must be dense in $\mathfrak{B}$ (according to the algebra's norm){\textemdash}this condition, in conjunction with injectivity, makes $\beta$ a *-isomorphism \cite[p. 243]{KaRi97}.

We are especially interested in minimal extensions, which capture limits of physical theories for accumulation points in the parameter space (such as $\hbar=0$). One can show that a minimal extension is guaranteed to exist for any accumulation point of interest.
\begin{theorem}
\label{thm:ex}
Let $\mathcal{A}_I = ((\mathfrak{A}_\hbar,\phi^I_\hbar)_{\hbar\in I},\mathfrak{A})$ be a $\mathit{UC}_b$ bundle of C*-algebras over a locally compact metric space $I$.  Suppose $\alpha: I\to J$ is a dense, isometric embedding.  Then there exists a minimal extension $\tilde{\mathcal{A}}_J$ of $\mathcal{A}_I$ associated with $\alpha$.
\end{theorem}

\noindent We proceed to prove Theorem \ref{thm:ex} through a series of lemmas.

\begin{lemma}\label{lem:quotient_1}
Suppose $\mathcal{A}_I$ and $\alpha:I\to J$ are as in Theorem 1.  Given any $j\in \overline{\alpha[I]}$, the set
\[\mathcal{K}_j := \left\{a\in\mathfrak{A}\ \middle|\ \lim_{\delta}\left\Vert\phi^I_{i_\delta}(a)\right\Vert_{i_\delta} = 0 \text{ for any net } \{i_\delta\}\text{ with } \alpha(i_\delta)\to j \right\}\]
is a closed two-sided ideal in $\mathfrak{A}$.
\end{lemma}

\begin{proof}
In what follows, we let $\{i_\delta\}$ denote an arbitrary net in $I$ with $\alpha(i_\delta)\to j$.

First, note that $\mathcal{K}_j$ is well-defined: for each $a\in\mathfrak{A}$, define $\Tilde{N_a} : \alpha[I] \to \mathbb{C}$ by
$$
\Tilde{N_a}(j) = N_a(\alpha^{-1} (j)),
$$
and note that the limit
$$
\lim_{\delta} \Tilde{N_a}(\alpha(i_\delta)) = \lim_{\delta} \left\Vert \phi^I_{i_\delta} (a) \right\Vert_{i_\delta}
$$
exists and is unique by the uniformly continuous extension theorem \cite[Lemma 3.11]{AlBo99}.  Now we go on to show that $\mathcal{K}_j$ is a closed, two-sided ideal.

First, $\mathcal{K}_j$ is a vector subspace of $\mathfrak{A}$; if $k,k'\in \mathcal{K}_j$ and $ x,y \in \mathbb{C}$, then
$$
\lim_{\delta} \left\Vert \phi^I_{i_\delta} (x  k + y k' ) \right\Vert_{i_\delta} \leq \lim_{\delta} x \left \Vert \phi^I_{i_\delta} (k) \right\Vert_{i_\delta}  + \lim_{\delta} y \left\Vert  \phi^I_{i_\delta} (k' ) \right\Vert_{i_\delta} = 0.
$$
Further, consider an arbitrary $k\in \mathcal{K}_j$ and $a\in \mathfrak {A}$; note that, because each norm is submultiplicative,
$$
\lim_{\delta} \left\Vert \phi^I_{i_\delta} (k \cdot a) \right\Vert_{i_\delta} = \lim_{\delta} \left\Vert \phi^I_{i_\delta} (k) \cdot \phi^I_{i_\delta} (a) \right\Vert_{i_\delta} \leq \lim_{\delta} \left\Vert \phi^I_{i_\delta} (k) \right\Vert_{i_\delta}  \cdot \left\Vert \phi^I_{i_\delta} (a) \right\Vert_{i_\delta} = 0,
$$
and similarly
$$
\lim_{\delta} \left\Vert \phi^I_{i_\delta} ( a \cdot k ) \right\Vert_{i_\delta} = 0,
$$
so $k\cdot a$ and $a \cdot k$ are both in $\mathcal{K}_j$.  Hence, $\mathcal{K}_j$ is a two-sided ideal. Finally, consider a net $\{ k_\lambda \}\subseteq \mathcal{K}_j$ that converges to $k$ in $\mathfrak{A}$ (which, recall, is equipped with the supremum norm). Pick some $\epsilon > 0$.
There is some $\lambda'$ such that for all $\lambda > \lambda'$, $\Vert k - k_\lambda \Vert < \epsilon/2$. Further, there is some $\delta'$ such that for all $\delta > \delta' $, $\Vert \phi^I_{i_\delta} (k_\lambda) \Vert_{i_\delta} < \epsilon/2$. Thus, for all $\delta >\delta'$ and $\lambda >\lambda'$,
\[\left \Vert \phi^I_{i_\delta} (k ) \right\Vert_{i_\delta} \leq \left\Vert\phi^I_{i_\delta}(k-k_\lambda)\right\Vert_{i_\delta} + \left\Vert\phi^I_{i_\delta}(k_\lambda)\right\Vert_{i_\delta} < \epsilon/2 + \epsilon/2 = \epsilon.\]
Since $\epsilon$ was arbitrary, we have $\lim_{\delta}  \Vert \phi^I_{i_\delta} (k ) \Vert_{i_\delta} =0$.  Hence, $k\in\mathcal{K}_j$, so $\mathcal{K}_j$ is closed.
\end{proof}

\begin{lemma}\label{lem:quotient_2}
With the definitions above, the canonical quotient $\mathfrak{A}/\mathcal{K}_j$ is a C*-algebra.
\end{lemma}

\begin{proof}
This follows immediately from Proposition 1.8.2 of \cite[pp. 20-21]{Di77}.
\end{proof}

\noindent
We now prove Theorem \ref{thm:ex}.

\begin{proof}[Proof of Theorem \ref{thm:ex}]
With Lemmas \ref{lem:quotient_1} and \ref{lem:quotient_2}, we define for each $\hbar\in J$ the C*-algebra $\Tilde{\mathfrak{A}}_\hbar:= \mathfrak{A}/\mathcal{K}_\hbar$.  Define $\phi^J_\hbar: \mathfrak{A}\to\Tilde{\mathfrak{A}}_\hbar$ for each $\hbar\in J$ as the canonical quotient map.  Then let $\tilde{\mathcal{A}}_J:=((\Tilde{\mathfrak{A}}_\hbar,\phi^J_\hbar)_{\hbar\in J},\mathfrak{A})$.  We will show that $\tilde{\mathcal{A}}_J$ is an extension of $\mathcal{A}_I$ associated with $\alpha$.

First, $\tilde{\mathcal{A}}_J$ is a $\mathit{UC}_b$ bundle of C*-algebras:
\begin{enumerate}[(i)]
\item Clearly, each map $\phi^J_\hbar$ is surjective by the definition of the quotient, which establishes fullness.
\item Given $f\in \mathit{UC}_b(J)$ and $a\in\mathfrak{A}$, we know $f\circ\alpha\in \mathit{UC}_b(I)$.  Hence, if we define $fa:= (f\circ\alpha)a$, then uniform completeness follows from the uniform completeness of $\mathcal{A}_I$.
\item Uniform continuity follows immediately from the construction with the definition of the quotient norm.
\end{enumerate}

Finally, we show that $\tilde{\mathcal{A}}_J$ is an extension of $\mathcal{A}_I$ associated with $\alpha$ via $\sigma:= (\alpha,\mathrm{id}_{\mathfrak{A}})$.  We know $\alpha$ is an isometric map, and $\mathrm{id}_{\mathfrak{A}}$ is a *-homomorphism.  Moreover, for any $a_1,a_2\in\mathfrak{A}$ and $\hbar\in I$, if $\phi_\hbar^I(a_1) = \phi_\hbar^I(a_2)$, then $(a_2-a_1)\in\mathcal{K}_{\alpha(\hbar)}$, and hence
\[\phi_{\alpha(\hbar)}^J(a_1) = \phi_{\alpha(\hbar)}^J(a_1 + (a_2-a_1)) = \phi_{\alpha(\hbar)}^J(a_2),\]
so $\sigma$ is a homomorphism.  Clearly, $\alpha$ and $\mathrm{id}_{\mathfrak{A}}$ are both injective, so $\sigma$ is a monomorphism.
\end{proof}

Now, we have established that minimal extensions always exist.  But in order to use this minimal extension to talk of, e.g., \emph{the} classical theory determined by a quantum theory, we further require a sense in which this extension is unique. It turns out that all minimal extensions associated with a dense embedding $\alpha$ are isomorphic{\textemdash}so it makes sense to talk of both \emph{the} minimal extension of a bundle and, e.g., \emph{the} classical theory that it defines.

\begin{theorem}
\label{thm:uni}
Let $\mathcal{A}_I = ((\mathfrak{A}_\hbar,\phi^I_\hbar)_{\hbar\in I},\mathfrak{A})$ be a $\mathit{UC}_b$ bundle of C*-algebras over a locally compact metric space $I$.  Suppose that $\mathcal{B}_J = ((\mathfrak{B}_\hbar,\phi^J_\hbar)_{\hbar\in J},\mathfrak{B})$ and $\mathcal{C}_J = ((\mathfrak{C}_\hbar,\psi^J_\hbar)_{\hbar\in J},\mathfrak{C})$ are two minimal extensions of $\mathcal{A}_I$ associated with a given dense, isometric embedding $\alpha: I\to J$. Then $\mathcal{B}_J$ and $\mathcal{C}_J$ are isomorphic as $\mathit{UC}_b$ bundles of C*-algebras.
\end{theorem}

\noindent We again require a preliminary lemma, which establishes general conditions under which a homomorphism of bundles generates isomorphisms of fiber algebras.
\begin{lemma}
\label{lem:fibmorph}
Let $\mathcal{A}_I$ and $\mathcal{B}_J$ be $\mathit{UC}_b$ bundles of C*-algebras and suppose $\sigma = (\alpha,\beta)$ is any homomorphism between them.  For each $\hbar\in I$, define the fiberwise map $\sigma_\hbar: \mathfrak{A}_\hbar\to\mathfrak{B}_{\alpha(\hbar)}$ by
\[\sigma_\hbar\!\left(\phi^I_\hbar(a)\right) = \phi^J_{\alpha(\hbar)}\left(\beta(a)\right)\]
for each $a\in\mathfrak{A}$. \begin{enumerate}[(i)]
\item $\sigma_\hbar$ is a *-homomorphism.
\item If $\sigma$ is a monomorphism, then $\sigma_\hbar$ is injective.
\item If $\beta$ is surjective, then $\sigma_\hbar$ is surjective.
\item If $\sigma$ is a monomorphism and $\beta$ is surjective, then $\sigma_\hbar$ is a *-isomorphism.  In particular,
\[\left\Vert\phi^I_\hbar(a)\right\Vert_\hbar = \left\Vert\phi^J_{\alpha(\hbar)}(\beta(a))\right\Vert_{\alpha(\hbar)}.\]
\end{enumerate}
\end{lemma}
\begin{proof} We prove each part in turn.

\begin{enumerate}[(i)]
\item  First, notice that the definition fully specifies $\sigma_\hbar$ because $\phi_\hbar^I$ is surjective.  Moreover, $\sigma_\hbar$ is well-defined because if $a_1,a_2\in\mathfrak{A}$ are such that $\phi_\hbar^I(a_1) = \phi_\hbar^I(a_2)$, then by the definition of a bundle homomorphism, we know $\sigma_\hbar(\phi_\hbar^I(a_1)) = \phi^J_{\alpha(\hbar)}(\beta(a_1)) = \phi^J_{\alpha(\hbar)}(\beta(a_2)) = \sigma_\hbar(\phi_\hbar^I(a_2))$.

Next, $\sigma_\hbar$ is linear: for any $x\in\mathbb{C}$ and $a_1,a_2\in\mathfrak{A}$,
\begin{align*}
    \sigma_\hbar(\phi_\hbar^I(a_1 + x\cdot a_2)) &= \phi_{\alpha(\hbar)}^J(\beta(a_1 + x\cdot a_2))\\ &= \phi_{\alpha(\hbar)}^J(\beta(a_1)) + x\cdot\phi_{\alpha(\hbar)}^J(\beta(a_2)) = \sigma_\hbar(\phi_\hbar^I(a_1)) + x\cdot\sigma_\hbar(\phi_\hbar^I(a_2)).
\end{align*}

Similarly, $\sigma_\hbar$ is multiplicative: for any $a_1,a_2\in\mathfrak{A}$,
\begin{align*}
    \sigma_\hbar(\phi_\hbar^I(a_1\cdot a_2)) &= \phi^J_{\alpha(\hbar)}(\beta(a_1\cdot a_2)) = \phi^J_{\alpha(\hbar)}(\beta(a_1))\cdot \phi^J_{\alpha(\hbar)}(\beta(a_2)) \\
    &=  \sigma_\hbar(\phi_\hbar^I(a_1))\cdot  \sigma_\hbar(\phi_\hbar^I(a_2)).
\end{align*}

Finally, $\sigma_\hbar$ is *-preserving: for any $a\in\mathfrak{A}$,
\begin{align*}
    \sigma_\hbar(\phi_\hbar^I(a^*)) = \phi_{\alpha(\hbar)}^J(\beta(a^*)) = \phi_{\alpha(\hbar)}^J(\beta(a))^* = \sigma_\hbar(\phi_\hbar^I(a))^*.
\end{align*}

\item We prove the contrapositive.  If $\sigma_\hbar$ is not injective, then there is some $a \in\mathfrak{A}$ for which $\phi_\hbar^I(a) \neq 0$, but $\sigma_\hbar(\phi_\hbar^I(a)) = \phi^J_{\alpha(\hbar)}(\beta(a)) = 0$.  Let $C^*(a)$ be the smallest C*-subalgebra of $\mathfrak{A}$ containing $a$ and consider the trivial bundle with fiber $C^*(a)$ over the one-element base space $\{*\}$.  Let $\iota: \{*\}\to I$ be defined by $\iota(*) = \hbar$, let $\beta_1:C^*(a)\to\mathfrak{A}$ be the natural inclusion, and let $\beta_2: C^*(a)\to\mathfrak{A}$ be the map that multiplies every element by 0.  Define $\sigma^1 = (\iota,\beta_1)$ and $\sigma^2 = (\iota,\beta_2)$.  Notice that $\sigma\circ \sigma^1 = \sigma\circ\sigma^2$ because for each $a'\in C^*(a)$, we know $\phi^J_{\alpha\circ\iota(*)}(\beta\circ\beta_1(a')) = 0 = \phi^J_{\alpha\circ\iota(*)}(\beta\circ\beta_2(a')).$  But clearly $\sigma^1\neq\sigma^2$ because $\sigma^1(a)= a\neq 0 = \sigma^2(a)$.  Thus, $\sigma$ is not a monomorphism.

\item Suppose $\beta$ is surjective. Then for any $b_{\alpha(\hbar)}\in\mathfrak{B}_{\alpha(\hbar)}$, since $\phi^J_{\alpha(\hbar)}$ is surjective, there is some $b\in\mathfrak{B}$ with $\phi^J_{\alpha(\hbar)}(b) = b_{\alpha(\hbar)}$.  But since $\beta$ is surjective, there is some $a\in\mathfrak{A}$ such that $\beta(a) = b$.  Hence,
\[\sigma_\hbar(\phi_\hbar^I(a)) = \phi^J_{\alpha(\hbar)}(\beta(a)) = \phi^J_{\alpha(\hbar)}(b) = b_{\alpha(\hbar)}.\]

\item This follows immediately from Theorem 4.1.8(iii) of \cite[p. 242]{KaRi97}.
\end{enumerate}
\end{proof}
\noindent From this lemma, Theorem \ref{thm:uni} quickly follows.
\begin{proof}[Proof of Theorem \ref{thm:uni}]
Since $\beta_{\mathcal{B}}$ and $\beta_{\mathcal{C}}$ are both injective and dense, it follows from Theorem 4.1.9 of \cite[p. 243]{KaRi97} that they are *-isomorphisms. Define $\beta: \mathfrak{B}\to\mathfrak{C}$ as  $\beta: = \beta_\mathcal{C}\circ\beta_\mathcal{B}^{-1}$.  Then consider $\sigma := (\mathrm{id}_J,\beta)$.  We need to show that $\sigma$ is an isomorphism of bundles.

Suppose that for $\hbar\in J$ and $b_1,b_2\in\mathfrak{B}$, we have $\phi^J_\hbar(b_1) = \phi^J_\hbar(b_2)$. We know that there is some net $\hbar_\delta\in I$ such that $\alpha(\hbar_\delta)\to \hbar$ in $J$.  Hence, by part (iii) of the previous lemma, we know that 
\begin{align*}
\norm{\phi_\hbar^J(\beta(b_1)) - \phi_\hbar^J(\beta(b_2))}_\hbar &= \norm{\phi_\hbar^J(\beta_\mathcal{C}\circ\beta_\mathcal{B}^{-1}(b_1)) - \phi_\hbar^J(\beta_\mathcal{C}\circ\beta_\mathcal{B}^{-1}(b_2))}_\hbar\\
&= \lim_\delta \norm{\phi_{\alpha(\hbar_\delta)}^J(\beta_\mathcal{C}\circ\beta_\mathcal{B}^{-1}(b_1)) - \phi_{\alpha(\hbar_\delta)}^J(\beta_\mathcal{C}\circ\beta_\mathcal{B}^{-1}(b_2))}_{\alpha(\hbar_\delta)}\\
&=\lim_\delta \norm{\phi^I_{\hbar_\delta}(\beta_\mathcal{B}^{-1}(b_1)) - \phi^I_{\hbar_\delta}(\beta_\mathcal{B}^{-1}(b_2))}_{\hbar_\delta}\\
&= \lim_\delta\norm{\phi_{\alpha(\hbar_\delta)}^J(b_1) - \phi_{\alpha(\hbar_\delta)}(b_2)}_{\alpha(\hbar_\delta)}\\
&= \norm{\phi_\hbar^J(b_1) - \phi_\hbar^J(b_2)}_\hbar = 0.
\end{align*}
Therefore, $\phi_\hbar^J(\beta(b_1)) = \phi_\hbar^J(\beta(b_2))$, and hence, $\sigma$ is a bundle homomorphism.  Since $\mathrm{id}_J$ and $\beta$ are both invertible, $\sigma$ is a bundle isomorphism.
\end{proof} 

\noindent
Theorems \ref{thm:ex} and \ref{thm:uni} together allow us to refer to $\tilde{\mathcal{A}}_J$ as \emph{the} minimal extension of a $\mathit{UC}_b$ bundle of C*-algebras $\mathcal{A}_I$ associated with a given dense embedding $\alpha:I\to J$. Moreover, Theorem \ref{thm:uni} allows us to refer to the algebra $\mathfrak{A}_j$ as \emph{the} algebra at the accumulation point $j\in (J \setminus \alpha[I])$. We will do so for the remainder of the paper. In particular, to represent the classical limit of a quantum theory we can set $I = (0,1]$, $J=[0,1]$, and look at the minimal extension associated with the inclusion map $\alpha:(0,1]\to[0,1]$.

So far, we have demonstrated only the existence and uniqueness of the C*-algebraic structure of the fiber $\mathfrak{A}_j$ over an accumulation point $j\in (J\setminus \alpha[I])$ for the base space $I$.  The limiting theory described by this algebra often carries additional structure, such as a privileged family of dynamical automorphisms, or a Lie bracket.  These structures can also be determined from corresponding dynamical automorphisms or Lie brackets on the fibers over the original base space $I$, and in the next section, we will proceed to demonstrate the functoriality of the determination of all of these structures.  Before proceeding, we note that there are corresponding existence and uniqueness results for the limiting dynamical structure and Lie bracket that one can extract from the discussion of \S\ref{sec:func}. Because they are somewhat trivial, we refrain from proving the existence and uniqueness of limiting dynamics and Lie brackets separately in this section; instead, we proceed directly to the discussion of functoriality. 

\section{Functoriality}
\label{sec:func}


Now we prove the functoriality of the procedure for constructing limits via $\mathit{UC}_b$ bundles of C*-algebras. We first define a functor $F$ implementing the extension of a $\mathit{UC}_b$ bundle of C*-algebras as in Theorems \ref{thm:ex} and \ref{thm:uni}.  Then, we define a functor $G$ representing the restriction from a bundle to the C*-algebra of the limit theory in the fiber over a point in the base space (e.g., $\hbar = 0$ in the classical limit).  We understand the limiting procedure to be represented by the functor $L$ obtained from the composition $G\circ F$.  This is accomplished in \S\ref{sec:C*}, whose upshot is an analysis of the determination of C*-algebraic structure on the limiting algebra.  The remainder of this section focuses on the determination of other kinds of structures by suitably adapting the previous results.  In \S\ref{sec:dyn}, we define categories in which objects carry the further structure of dynamical automorphisms, and we show that this structure functorially determines dynamical automorphisms on the limiting algebra.  In \S\ref{sec:lie}, we specialize to the primary example of the classical limit of quantum theory, in which the Lie bracket defined by the commutator at $\hbar>0$ corresponds to the Lie bracket defined by the Poisson structure at $\hbar = 0$.  We that show the commutator functorially determines the classical Poisson bracket on the limiting algebra.

\subsection{C*-algebraic structure}
\label{sec:C*}

We want to understand the extension from the base space $(0,1]$ to $[0,1]$ as an instance of a general procedure applicable to other base spaces.  To that end, we understand $[0,1]$ as the one-point compactification of $(0,1]$.  We will rely in this section only on the fact that the one-point compactification is itself functorial, so we proceed with some generality.

To start, suppose one has a non-empty collection $\mathcal{I}$ of locally compact metric spaces as the base spaces of possible bundles.  Suppose further that one has an assignment to each $I\in\mathcal{I}$ of a further locally compact metric space $C(I)$ and a dense isometric embedding $\alpha_I:I\to C(I)$.  We will require that the assignment $I\mapsto (C(I),\alpha_I)$ is \emph{functorial}.  By this, we mean to require that all suitable functions between spaces $I$ and $J$ in $\mathcal{I}$ can be continuously extended to functions between $C(I)$ and $C(J)$, in a way that behaves well with the embeddings $\alpha_I$ and $\alpha_J$.  It is well known that in general we cannot hope to extend arbitrary continuous functions between $I$ and $J$ to the larger spaces $C(I)$ and $C(J)$; instead, we must restrict to the \emph{proper} maps between $I$ and $J$, which are the functions between base spaces that are suitable to be extended.  Recall that a map $\alpha: I\to J$ between topological spaces $I$ and $J$ is called \emph{proper} if for each compact set $K\subseteq J$, the inverse image $\alpha^{-1}[K]\subseteq I$ is compact.  The restriction to proper maps is necessary in order to put the one-point compactification of the base space in the purview of our extension results, as the one-point compactification is functorial only when restricted to proper maps.

\begin{definition}
An assignment $I\mapsto (C(I),\alpha_I)$ is \emph{functorial} if
\begin{enumerate}[(i)]
\item for each proper metric map $f: I\to J$, there is a map $C(f): C(I)\to C(J)$ such that $C(f)\circ \alpha_I = \alpha_J\circ f$, or in other words, the following diagram commutes:
\begin{equation*}
\begin{tikzcd}
I \arrow{r}{f} \arrow{d}[swap]{\alpha_I}
& J\arrow{d}{\alpha_J} \\
C(I) \arrow{r}{C(f)}
& C(J)
\end{tikzcd}
\end{equation*}
\item if $f:I\to J$ and $g: J\to K$ are proper metric maps, then $C(g\circ f) = C(g)\circ C(f)$.
\end{enumerate}
\end{definition}

\noindent With a functorial assignment $I\mapsto (C(I),\alpha_I)$, we can simultaneously consider extending each bundle over $I$ along $\alpha_I$ to $C(I)$.  Such an assignment gives a standard for taking \emph{the same kind} of extension of all of the bundles considered.  Moreover, the functoriality of the assignment establishes that the structure of $C(I)$ is determined by the structure of $I$.

For example, given any locally compact metric space $I$, let $\dot{I} = I\cup \{0\}$ denote its one-point compactifiation (we in general denote the added point by $0$ because we have in mind the extension from $(0,1]$ to $[0,1]$) with inclusion $\iota: I\to\dot{I}$  \cite[p. 169, Theorem 3.5.11]{En89}.  Let $\mathcal{I}$ be the collection of locally compact, non-compact metric spaces whose one-point compactification $\dot{I}$ is metrizable and whose embedding $\iota:I\to\dot{I}$ is an isometric map.  This includes $I = (0,1]$, but rules out, for instance, $I = \mathbb{R}$.  Then the assignment $C(I):= \dot{I}$ and $\alpha_I = \iota$ for each $I\in\mathcal{I}$ provides us with a starting point we can use to take the extension of any bundle over a base space $I\in\mathcal{I}$.  It is well known that the one-point compactification is functorial in the sense outlined above for proper metric maps.  In this case, there is a strong sense in which the one-point compactification yields the same kind of extension for each of the bundles considered, and the structure of the one-point compactification $\dot{I}$ is determined by the structure of the original space $I$.

In what follows, we will have the one-point compactification in mind, but for now we will only suppose that we have some functorial assignment or other of an enlarged space $C(I)$ and a dense isometric embedding $\alpha_I$ to each $I\in\mathcal{I}$.  Given such an assignment, we will restrict attention to a subcategory $\mathbf{UBunC^*Alg}_\mathcal{I}$ of $\mathbf{UBunC^*Alg}$ consisting solely of the $\mathit{UC}_b$ bundles of C*-algebras over base spaces $I\in\mathcal{I}$.  Moreover, we restrict attention to morphisms that are proper maps between base spaces.
\begin{definition}
The category $\mathbf{UBunC^*Alg}_\mathcal{I}$ consists in:
\begin{itemize}
    \item \textit{objects}: $\mathit{UC}_b$ bundles of C*-algebras whose base space $I$ belongs to $\mathcal{I}$;
    \item \textit{morphisms}: homomorphisms $\sigma = (\alpha,\beta): \mathcal{A}_I\to\mathcal{B}_J$ between $\mathit{UC}_b$ bundles of C*-algebras with $I,J\in\mathcal{I}$, where $\alpha$ is proper.
\end{itemize}
\end{definition}

\noindent Similarly, letting $C(\mathcal{I}) := \{C(I)\ |\ I\in\mathcal{I}\}$, the category $\mathbf{UBunC^*Alg}_{C(\mathcal{I})}$ is the subcategory of $\mathbf{UBunC^*Alg}$ consisting of the $\mathit{UC}_b$ bundles of C*-algebras over base spaces $C(I)\in C(\mathcal{I})$ with homomorphisms between them that act as proper maps between base spaces.  We now use the data contained in the functor $C$ to define an extension functor $F:\mathbf{UBunC^*Alg}_\mathcal{I}\to\mathbf{UBunC^*Alg}_{C(\mathcal{I})}$.

We define $F$ on objects and morphisms by
\begin{gather}\label{eq:functor_f}
    \begin{aligned}
    \mathcal{A}_I \, &\mapsto \, \tilde{\mathcal{A}}_{C(I)} \\
    (\alpha,\beta) \, &\mapsto \, (C(\alpha),\beta)
\end{aligned}
\end{gather}
where $C(I)$ is the given enlarged base space assigned to $I$ with dense isometric embedding $\alpha_I: I\to C(I)$, and $\tilde{\mathcal{A}}_{C(I)}$ is the unique minimal extension of $\mathcal{A}_I$ associated with $\alpha_I$ guaranteed by Theorems \ref{thm:ex} and \ref{thm:uni}. Similarly, $(\alpha,\beta)$ is a morphism between $\mathit{UC}_b$ bundles of C*-algebras $\mathcal{A}_I = ((\mathfrak{A}_\hbar,\phi^I_\hbar)_{\hbar\in I},\mathfrak{A})$ and $\mathcal{B}_J = ((\mathfrak{B}_\hbar,\phi^J_\hbar)_{\hbar\in J},\mathfrak{B})$, and $C(\alpha)$ is the functorial extension of $\alpha: I\to J$ to $C(\alpha): C(I)\to C(J)$.  In our construction of the extension $\tilde{\mathcal{A}}_{C(I)}$ along $\alpha_I$, we defined the C*-algebra of continuous sections of $\tilde{\mathcal{A}}_{C(I)}$ to be identical with the C*-algebra $\mathfrak{A}$ of continuous sections in $\mathcal{A}_I$; we only needed to extend their assignments by specifying new evaluation maps $\phi_\hbar^{C(I)}$ for $\hbar\in \big(C(I)\setminus\alpha_I[I]\big)$. Hence, it makes sense to leave $\beta:\mathfrak{A}\to \mathfrak{B}$ unchanged in the extension.  It is simple to check that this assignment respects the composition of morphisms, so $F$ is indeed a functor.

\begin{proposition}\label{prop:functor_f}
Given any functorial assignment $C$ of dense isometric embeddings $\alpha_I: I\to C(I)$ to locally compact metric spaces $I\in\mathcal{I}$, the assignment $F$ defined by (\ref{eq:functor_f}) is a functor.
\end{proposition}

\noindent This shows a sense in which, relative to the embeddings encoded in $C$, the construction of bundle extensions is natural.  It requires no further information or structure.  Moreover, when $I = (0,1]$ and $C$ is the one-point compactification, then the extension of a bundle over $I$ to a bundle over $C(I) = [0,1]$ is an instance of this natural construction.  This is precisely the form of our extension of bundles for $\hbar>0$ to the classical limit at $\hbar = 0$.

We encode the remainder of the limit in the restriction from the base space $[0,1]$ to the structure of the limiting algebra at $\hbar = 0$.  We now restrict ourselves to the case where $\mathcal{I}_{1}$ is the collection of all locally compact, non-compact metric spaces whose one-point compactification $C_{1}(I) := \dot{I}$ is such that the canonical embedding $\alpha_I := \iota: I\to C_1(I)$ is an isometric map.  We will use this fixed $C = C_1$ and $\mathcal{I} = \mathcal{I}_1$ for the remainder of this section.  Since there is a unique point $0_I\in (C_1(I)\setminus\alpha_I[I])$, the restriction of a bundle over $C_1(I)$ to the fiber over $\hbar = 0_I$ is a natural construction.  We now define the restriction functor to the fiber over $0_I$.

The restriction functor will have the category $\mathbf{UBunC^*Alg}_{C_1(\mathcal{I}_1)}$ as its domain.  The codomain category of the restriction functor is simply the category of C*-algebras.
\begin{definition}
The category $\mathbf{C^*Alg}$ consists in:
\begin{itemize}
    \item \textit{objects}: C*-algebras;
    \item \textit{morphisms}: *-homomorphisms.
\end{itemize}
\end{definition}

Our restriction functor $G:\mathbf{UBunC^*Alg}_{C_1(\mathcal{I}_1)}\to \mathbf{C^*Alg}$ is now defined through the following actions on objects and morphisms:
\begin{gather}\label{eq:functor_g}
    \begin{aligned}
        \tilde{\mathcal{A}}_{C_1(I)} = \left( \left(\mathfrak{A}_\hbar,\phi^{C_1(I)}_\hbar\right)_{\hbar\in C_1(I)},\mathfrak{A}\right) \, &\mapsto \, \mathfrak{A}_{0_I} \\
        \sigma = \left(\alpha,C_1(\alpha),\beta\right) \, &\mapsto \, \sigma_0,
    \end{aligned}
\end{gather}
where we specify
\begin{equation}\label{eq:sigma0_def}
    \sigma_0\!\left(\phi^{C_1(I)}_{0_I}(a)\right) := \phi^{C_1(J)}_{0_J}\left(\beta(a)\right)
\end{equation}
for any $a\in \mathfrak{A}$.  Since $C_1(\alpha)$ must map $0_I$ to $0_J$ (which follows from the universal property of the one-point compactification), we know that $\sigma_0$ is well-defined and a *-homomorphism (by Lemma \ref{lem:fibmorph}). We can understand the morphism $\sigma_0$ to be the classical limit in $\mathbf{C^*Alg}$ of the morphism $\sigma$ in $\mathbf{UBunC^*Alg}_{C_1(\mathcal{I}_1)}$.  Furthermore, it is again simple to check that this assignment respects composition of morphisms, so it follows that $G$ is a functor.
\begin{proposition}\label{prop:functor_g}
When $C_1$ is the one-point compactification, the assignment $G$ defined by (\ref{eq:functor_g}) is a functor.
\end{proposition}

This shows a sense in which the restriction of a bundle over $C_1(I)$ to the fiber over $0_I$ is natural given the structure of the category $\mathbf{UBunC^*Alg}_{C_1(\mathcal{I}_1)}$.  And finally, we can define a functor $L := G\circ F$, which provides a natural construction of the limiting C*-algebra.

\subsection{Dynamical structure}
\label{sec:dyn}

In this section, we show that the limiting construction is functorial even when we consider the further structure of dynamics.  We understand dynamics on a C*-algebra to be encoded in a one-parameter family of automorphisms.  Our construction of a limiting C*-algebra for a bundle allows us to obtain limiting dynamics, under the condition that the dynamics scales continuously with the limiting parameter, which we make precise as follows.

Suppose again that $\mathcal{A}_I = ((\mathfrak{A}_\hbar,\phi_\hbar^I)_{\hbar\in I},\mathfrak{A})$ is a $\mathit{UC}_b$ bundle of C*-algebras and allow $(\tau_{t;\hbar})_{t\in\mathbb{R}}$ to denote the dynamics on each fiber algebra $\mathfrak{A}_\hbar$. To enforce that these dynamics scale continuously with $\hbar$, we require that they lift to an automorphism group on the algebra of sections. That is, we require the existence of a one-parameter automorphism group $(\tau_{t})_{t\in\mathbb{R}}$ on $\mathfrak{A}$ satisfying
\begin{equation}\label{eq:dynamics_continuous_scale}
   \tau_{t;\hbar}\circ\phi_\hbar^I = \phi_\hbar^I\circ \tau_t\ \ \ \ \text{ for all } \hbar\in I,\ t\in\mathbb{R}.
\end{equation}
This is just the requirement that we are considering (in some sense) the ``same" dynamics at different scales.  We now define objects that contain such dynamics.

\begin{definition}
A \emph{dynamical bundle of C*-algebras} is a $\mathit{UC}_b$ bundle of C*-algebras $\mathcal{A}_I = ((\mathfrak{A}_\hbar,\phi_\hbar^I)_{\hbar\in I},\mathfrak{A})$ over a locally compact, non-compact metric space $I\in\mathcal{I}$ with a one-parameter group of automorphisms $(\tau^\mathcal{A}_t)_{t\in\mathbb{R}}$ on $\mathfrak{A}$.
\end{definition}
Further, we define a category of dynamical bundles of C*-algebras that encodes the extra structure of the dynamics.  For the moment, we return to the case of some generality and suppose only that $\mathcal{I}$ is a non-empty collection of locally compact metric spaces that serve as the base spaces of possible bundles, and that $I\mapsto (C(I),\alpha_I)$ is a functorial assignment of extended locally compact metric spaces $C(I)$ with dense embeddings $\alpha_I: I\to C(I)$.  As before, we now restrict attention to categories of dynamical bundles whose base spaces belong to this collection.

\begin{definition}
The category $\mathbf{DBunC^*Alg}_{\mathcal{I}}$ consists in:
\begin{itemize}
    \item \textit{objects}: dynamical bundles of C*-algebras $(\mathcal{A}_I,(\tau_t^\mathcal{A})_{t\in\mathbb{R}})$ with $I\in\mathcal{I}$;
    \item \textit{morphisms}: homomorphisms $\sigma = (\alpha,\beta):\mathcal{A}_I,\to\mathcal{B}_J$ between $\mathit{UC}_b$ bundles of C*-algebras with $I,J\in\mathcal{I}$, where $\alpha$ is proper and and for each $t\in\mathbb{R}$ \[\tau^\mathcal{B}_t\circ\beta = \beta\circ\tau^\mathcal{A}_t.\]
\end{itemize}
\end{definition}
\noindent Note that a morphism in $\mathbf{DBunC^*Alg}$ is just a morphism in $\mathbf{UBunC^*Alg}$ that furthermore preserves the structure of the dynamics.

We will first define an extension functor $F_D$, and then define a restriction functor $G_D$; the classical limit will again be understood as the composition $L_D = G_D\circ F_D$.

The extension functor $F_D:\mathbf{DBunC^*Alg}_{\mathcal{I}}\to\mathbf{DBunC^*Alg}_{C(\mathcal{I})}$ is defined to act on objects and morphisms as follows:
\begin{gather}\label{eq:functor_fd}
    \begin{aligned}
        \mathcal{A}_I = \left(\left(\mathfrak{A}_\hbar,\phi^I_\hbar\right)_{\hbar\in I},\mathfrak{A},(\tau^\mathcal{A}_t)_{t\in\mathbb{R}})\right) \, &\mapsto \, \left(F(\mathcal{A}_I),(\tau^\mathcal{A}_t)_{t\in\mathbb{R}} \right) \\
        \sigma \, &\mapsto \, F(\sigma).
    \end{aligned}
\end{gather}
$F_D$ sends a dynamical bundle $\mathcal{A}_I$ to its unique extension associated with $\alpha_I$ with the same dynamics; and $F_D$ extends morphisms just as $F$ does.  It is immediate that the extension of any dynamics-preserving morphism will also preserve dynamics, so that given a morphism $\sigma$ in $\mathbf{DBunC^*Al}g_{\mathcal{I}}$, $F(\sigma)$ is indeed a morphism in $\mathbf{DBunC^*Alg}_{C(\mathcal{I})}$.  This assignment is a functor by Proposition \ref{prop:functor_f}.

\begin{proposition}\label{prop:functor_F_D}
Given any functorial assignment $C$ of dense isometric embeddings $\alpha_I:I\to C(I)$ to locally compact metric spaces $I\in\mathcal{I}$, the assignment $F_D$ is a functor. 
\end{proposition}

Next, we define a restriction functor $G_D$.  We specify now that $C = C_1$ will be the one-point compactification, as above, and so $\mathcal{I} = \mathcal{I}_1$ will again consist in those locally compact, non-compact metric spaces whose embeddings in their one-point compactifications are isometric maps.  The codomain of our restriction functor will be a category of algebras of quantities of classical theories, now with dynamics, defined as follows.

\begin{definition}
The category $\mathbf{DC^*Alg}$ consists in:
\begin{itemize}
    \item \textit{objects}: pairs $(\mathfrak{A},(\tau^\mathfrak{A}_t)_{t\in\mathbb{R}})$, where $\mathfrak{A}$ is a C*-algebra and $(\tau^\mathfrak{A}_t)_{t\in\mathbb{R}}$ is a one-parameter group of automorphisms on $\mathfrak{A}$;
    \item \textit{morphisms}: morphisms $\sigma_0:(\mathfrak{A},(\tau_t^\mathfrak{A})_{t\in\mathbb{R}})\to(\mathfrak{B},(\tau_t^\mathfrak{B})_{t\in\mathbb{R}})$, each of which is a *-homomorphism $\sigma_0:\mathfrak{A}\to\mathfrak{B}$ such that for each $t\in\mathbb{R}$,
    \[\tau_t^\mathfrak{B}\circ\sigma_0 = \sigma_0\circ\tau_t^\mathfrak{A}.\]
\end{itemize}
\end{definition}
With this category in hand, we define the functor $G_D: \mathbf{DBunC^*Alg}_{C_1(\mathcal{I}_1)}\to\mathbf{DC^*Alg}$ as follows.  Given any dynamical bundle of C*-algebras $\mathcal{A}_I = \left(\left(\mathfrak{A}_\hbar,\phi^I_\hbar\right)_{\hbar\in I},\mathfrak{A},(\tau^\mathcal{A}_t)_{t\in\mathbb{R}}\right)$, one can restrict the dynamics to a fiber, defining for each $t\in\mathbb{R}$ an automorphism $\tau^\mathcal{A}_{t;\hbar}$ of $\mathfrak{A}_\hbar$ by
\begin{equation}
\label{eq:restdyn}
    \tau^\mathcal{A}_{t;\hbar}\!\left(\phi_\hbar^I(a)\right) := \phi_\hbar^I\!\left(\tau^\mathcal{A}_t(a)\right),
\end{equation}
for all $a\in\mathfrak{A}$.  We now use this definition in the case where the base space is $C_1(I)$ and we focus on the the fiber over $0_I\in C_1(I)$.    Note that the way we construct limiting dynamics $\tau_{t;0_I}^\mathcal{A}$ at $\hbar = 0_I$ is the same as the way that $G$ acts on morphisms because the dynamics are represented by *-homomorphisms (compare Equations (\ref{eq:sigma0_def}) and (\ref{eq:restdyn})).  In other words, we have $\tau_{t;0_I}^\mathcal{A} = G(\tau_t^\mathcal{A})$.  Hence, we define $G_D$ to act on objects by
\begin{gather}
\label{eq:fun_gd}
\begin{aligned}
    \tilde{A}_{C_1(I)} = \Big((\mathfrak{A}_\hbar,\phi^{C_1(I)}_\hbar)_{\hbar\in C_1(I)},\mathfrak{A},(\tau_t^\mathcal{A})_{t\in\mathbb{R}}\Big)&\mapsto \Big(\mathfrak{A}_{0_I},(\tau^\mathcal{A}_{t;0_I})_{t\in\mathbb{R}}\Big)\\
    \sigma & \mapsto G(\sigma)
\end{aligned}
\end{gather}
It follows immediately from  Proposition \ref{prop:functor_g} that for any morphism $\sigma$ in $\mathbf{DBunC^*Alg}_{C(\mathcal{I})}$, $G_D(\sigma)$ is a morphism in $\mathbf{DBunC^*Alg}$, so $G_D$ is well-defined and a functor.

\begin{proposition}\label{prop:ForgetG_D}
When $C_1$ is the one-point compactification, the assignment $G_D$ defined by (\ref{eq:fun_gd}) is a functor.
\end{proposition}

Finally, we can define the classical limit functor for dynamics as the composition of the extension functor and restriction functor.  Thus, we define $L_D:= G_D\circ F_D$. It follows immediately that $L_D$ is a functor, showing one sense in which the construction of limiting dynamics is natural.

We close this section with two remarks about the limiting dynamics determined by $L_D$.  First, it follows trivially that the limiting dynamics exists and is unique (up to morphism in $\mathbf{DC^*Alg}$).  Second, it follows from results in \cite{La98b} that the dynamics defined by certain classes of Hamiltonians in quantum theories determine by $L_D$ a corresponding classical Hamiltonian dynamics (For bounded Hamiltonians, this follows from Proposition 2.7.1, p.138; for Hamiltonians at most quadratic in $p$ and $q$, this follows from Corollary 2.5.2, p. 141).  Thus, the functor $L_D$ defines dynamics that are not just natural, but also reproduce the physical dynamics that one expects upon convergence in the $\hbar\to 0$ limit.

\subsection{Lie bracket}
\label{sec:lie}

In this section, we restrict attention to the classical limit of quantum theory, and so we leave behind other kinds of limits that may be captured by continuous bundles of C*-algebras.  We will thus restrict attention to the kinds of bundles used for quantization and the classical limit, but we here define a category that contains as objects slight generalizations of the bundles over $(0,1]$ generated by strict quantizations.  Our generalization allows one to include bundles over base spaces that are not themselves identified with $(0,1]$, i.e., possibly more abstract ``parameters".  We will show that in these kinds of bundles, the commutator Lie bracket of the fiber C*-algebras representing quantum theories functorially determines the Poisson bracket of the fiber C*-algebra representing a classical theory.

To define the relevant categories, we again restrict attention to $C_1$ as the one-point compactification and $\mathcal{I}_1$ as the collection of locally compact, non-compact topological spaces whose embedding in their one-point compactification is an isometric map.  For ease of notation, given a base space $I\in\mathcal{I}_1$, we let
\begin{equation}\label{eq:hbar_distance}
    \abs{\hbar}_I:= d_{C_1(I)}\!\left(\alpha_I(\hbar),0_I\right)
\end{equation}
denote the distance in the one-point compactification $C_1(I)= \dot{I}$ of $\hbar$ (as it is embedded by $\alpha_I$) from $0_I$.  For example, in the base space $I = (0,1]$ with $\dot{I} = [0,1]$, we have $\abs{\hbar}_I = \abs{\hbar - 0} = \hbar$ for each $\hbar\in I$.  Now we define a class of bundles generalizing those produced by strict quantizations, as follows.

\begin{definition}
\label{def:pqbun}
A $\mathit{UC}_b$ bundle of C*-algebras $\mathcal{A}_I = ((\mathfrak{A}_\hbar,\phi^I_\hbar)_{\hbar\in I},\mathfrak{A})$ over a locally compact, non-compact metric space $I\in\mathcal{I}$ is called a \emph{post-quantization bundle} if there is a subset $P_\mathcal{A}\subseteq \mathfrak{A}$ such that \begin{enumerate}[(i)]
\item $P_\mathcal{A}$ is norm dense in $\mathfrak{A}$;
\item for each pair $a,b\in P_\mathcal{A}$, there is a $c_{a,b}\in P_\mathcal{A}$ such that $\phi^I_\hbar(c_{a,b}) = \frac{i}{\abs{\hbar}_I}[\phi^I_\hbar(a),\phi^I_\hbar(b)]$ for all $\hbar\in I$; and
\item for each pair $a,b\in \mathfrak{A}$, $\lim_{\hbar\to 0_I}\norm{[\phi^I_\hbar(a),\phi^I_\hbar(b)]} = 0$.
\end{enumerate}
\end{definition}
\noindent Condition (iii) guarantees that the classical limit algebra $\mathfrak{A}_0$ is commutative.  On the other hand, conditions (i) and (ii) allow us to define a norm dense Poisson algebra of $\mathfrak{A}_0$.

Given a post-quantization bundle, we can define a Poisson bracket $\{\cdot,\cdot\}_\mathcal{A}$ on the dense subset $\phi_{0_I}^{C_1(I)}[P_\mathcal{A}]\subseteq \mathfrak{A}_0$ by
\begin{equation}
\label{eq:Poiss}
    \left\{\phi^{C_1(I)}_{0_I}(a),\phi^{C_1(I)}_{0_I}(b)\right\}_\mathcal{A} := \phi_0(c_{a,b})
\end{equation}
for every $a,b\in P_\mathcal{A}$.

We will make the structure of post-quantization bundles precise by defining a corresponding category of bundles.  To do so, we consider morphisms between bundles that preserve the relevant additional structure.  We will say that a morphism $\sigma = (\alpha,\beta)$ between post-quantization bundles $\mathcal{A}_I = ((\mathfrak{A}_\hbar,\phi_\hbar^I)_{\hbar\in I},\mathfrak{A},P_{\mathcal{A}})$ and $\mathcal{B}_J = ((\mathfrak{B}_\hbar,\phi_\hbar^J)_{\hbar\in J},\mathfrak{B},P_{\mathcal{B}})$ is \emph{smooth} if $\beta[P_{\mathcal{A}}] \subseteq P_{\mathcal{B}}$.  Since the subalgebra $P_\mathcal{A}$ is the collection of quantities whose commutators scale appropriately with $\hbar$, the smoothness condition guarantees that the image of these quantities under $\beta$ belongs to $P_\mathcal{B}$ and hence, have commutators that scale appropriately with $\hbar$ as well.  This is a pre-condition for the morphism to respect the scaling of commutators with $\hbar$, and for the classical limit of a morphism to preserve the Poisson bracket.

Further, we will say that a morphism $\sigma = (\alpha,\beta)$ between post-quantization bundles $\mathcal{A}_I = ((\mathfrak{A}_\hbar,\phi_\hbar^I)_{\hbar\in I},\mathfrak{A},P_{\mathcal{A}})$ and $\mathcal{B}_J = ((\mathfrak{B}_\hbar,\phi_\hbar^J)_{\hbar\in J},\mathfrak{B},P_{\mathcal{B}})$ is
\emph{second-order} if there is some constant $K\geq 0$ such that $\Big\vert\frac{1}{\abs{\hbar}_I} - \frac{1}{\abs{\alpha(\hbar)}_J}\Big\vert \rightarrow K$ as $\hbar\to 0_I$.  This is a generalization of the special case of most interest to us for the physics of the classical limit.  The special case we have in mind consists in bundles over the same base space $I = (0,1]$ (with $C_1(I) = [0,1]$ and $0_I = 0$) and with a re-scaling map $\alpha: [0,1]\to[0,1]$.  Consider an example where $\alpha$ has a power series expansion (in a neighborhood of $\hbar = 0$) of the form
\begin{align}
    \alpha(\hbar) = \sum_{n=0}^\infty a_n\hbar^n.
\end{align}
In this case, the name ``second-order" is meant to suggest that $\alpha$ agrees with the identity map up to first order in $\hbar$ and differs only for powers $n>2$.  In other words, consider the case where $a_0 = 0$ and $a_1 = 1$.  Then
\begin{align}
    \lim_{\hbar\to 0}\left\vert \frac{1}{\hbar} - \frac{1}{\alpha(\hbar)}\right\vert &=\lim_{\hbar\to 0}\left\vert \frac{\hbar + \sum_{n=2}^\infty a_n\hbar^n - \hbar}{\sum_{n=1}^\infty a_n\hbar^{n+1}}\right\vert = \lim_{\hbar\to 0}\left\vert \frac{a_2\hbar^2 + \sum_{n=3}^\infty a_n\hbar^n}{\hbar^2 + \sum_{n=2}^\infty a_n\hbar^{n+1}}\right\vert = \left\vert a_2\right\vert.
\end{align}
Hence, in this special case where $\alpha$ has the form
\begin{align}
     \alpha(\hbar) = \hbar + \sum_{n=2}^\infty a_n\hbar^n,
\end{align}
we know $\alpha$ is second-order according to our definition with $K = \abs{a_2}$.  A rough physical interpretation of the second-order condition is that the scaling behavior of $\hbar$ in $I$ agrees enough with $\alpha(\hbar)$ in $J$ in the limit as one approaches $0_I$ in $I$ and $0_J$ in $J$.  In other words, the condition requires that $\alpha$ preserve at least some metrical structure of the base spaces.

These definitions allow us to define a category that slightly generalizes the continuous bundles obtained from strict deformation quantizations.

\begin{definition}
The category $\mathbf{PQBunC^*Alg}_{\mathcal{I}_1}$ consists in:
\begin{itemize}
    \item \textit{objects}: post-quantization bundles whose base space $I$ belongs to $\mathcal{I}_1$.
    \item \textit{morphisms}: smooth, second-order morphisms of $\mathit{UC}_b$ bundles of C*-algebras whose map between base spaces is proper.
\end{itemize}
\end{definition}

\noindent As noted in \S\ref{sec:quant} and detailed in \ref{app:quanttobun}, one can construct a unique bundle of C*-algebras from any strict deformation quantization satisfying modest conditions.  $\mathbf{PQBunC^*Alg}_{\mathcal{I}_1}$ can be thought of as a subcategory of $\mathit{UC}_b$ bundles of C*-algebras determined by this construction.

Similarly, we define a category of classical theories obtained through the classical limit.  This category encodes additional Poisson structure of classical phase spaces.
\begin{definition}
The category $\mathbf{Class}$ consists in:
\begin{itemize}
    \item \textit{objects}: pairs $(\mathfrak{A}_0,(P^0_\mathcal{A},\{\cdot,\cdot\}_\mathcal{A}))$, where $\mathfrak{A}_0$ is a commutative C*-algebra and $P^0_\mathcal{A}$ is a norm-dense Poisson algebra with bracket $\{\cdot,\cdot\}_\mathcal{A}$;
    \item \textit{morphisms}: morphisms $\sigma_0: (\mathfrak{A}_0,(P^0_\mathcal{A},\{\cdot,\cdot\}_\mathcal{A}))\to (\mathfrak{B}_0,(P^0_\mathcal{B},\{\cdot,\cdot\}_\mathcal{B}))$, each of which is a *-homomorphism $\sigma_0: \mathfrak{A}_0\to\mathfrak{B}_0$ satisfying
    \begin{enumerate}[(i)]
    \item $\sigma_0[P^0_\mathcal{A}]\subseteq P^0_\mathcal{B}$; and \item the restriction of $\sigma_0$ to $P^0_\mathcal{A}$ is a Poisson morphism, i.e., for each $A,B\in P^0_\mathcal{A}$,
    \begin{align*}
        \{\sigma_0(A),\sigma_0(B)\}_\mathcal{B} = \sigma_0\!\left(\{A,B\}_\mathcal{A}\right).
    \end{align*}
    \end{enumerate}
\end{itemize}
\end{definition}
\noindent The restrictions on morphisms in $\mathbf{Class}$ again guarantee that smooth quantites get mapped to smooth quantities, and that the morphisms define the further classical structure of the Poisson bracket on the phase space.  The canonical example of an object in $\mathbf{Class}$ is a pair $(C_0(M),(C_c^\infty(M),\{\cdot,\cdot\})$ for a Poisson manifold $(M,\{\cdot,\cdot\})$.

We define a classical limit functor $L_P$ analogous to $L$.  Define the functor $L_P$ from  $\mathbf{PQBunC^*Alg}_{\mathcal{I}_1}$ to $\mathbf{Class}$ on objects and arrows by
\begin{gather}\label{eq:functor_lp}
    \begin{aligned}
        \mathcal{A}_I = \left(\left(\mathfrak{A}_\hbar,\phi^I_\hbar\right)_{\hbar\in I},\mathfrak{A},P_\mathcal{A}\right) \, &\mapsto \, \left(L(\mathcal{A}_I),\left(\phi^{C_1(I)}_{0_I}[P_\mathcal{A}],\{\cdot,\cdot\}_\mathcal{A}\right)\right) \\
        \sigma \, &\mapsto \, L(\sigma) = \sigma_0,
    \end{aligned}
\end{gather}
where $\{\cdot,\cdot\}_\mathcal{A}$ is defined by (\ref{eq:Poiss}).  As mentioned, it follows trivially that the Poisson bracket $\{\cdot,\cdot\}_\mathcal{A}$ both exists and is the unique Lie bracket (up to morphisms in $\mathbf{Class}$) on the fiber $\mathfrak{A}_0$ that continuously extends the commutator for $\hbar >0$.  The following guarantees $L_P$ is well-defined and a functor.

\begin{proposition}
\label{prop:Poisson}
Suppose $\mathcal{A}_I$ and $\mathcal{B}_J$ are objects in $\mathbf{PQBunC^*Alg}_{\mathcal{I}_1}$ and $\sigma:\mathcal{A}_I\to\mathcal{B}_J$ is a morphism (i.e., smooth and second-order).  Then $L_P(\sigma)$ is a morphism in $\mathbf{Class}$ (i.e., restricts to a Poisson morphism).  Hence, $L_P$ is a functor.
\end{proposition}

\begin{proof}
Suppose $\sigma: \mathcal{A}_I\to\mathcal{B}_J$ is a morphism in $\mathbf{PQBunC^*Alg}$, where $\sigma = (\alpha,\beta)$.  We need to show that for each $a,b\in P_\mathcal{A}$,
    \[\left\{L_P(\sigma)\!\left(\phi^I_{0_I}(a)\right),L_P(\sigma)\!\left(\phi^I_{0_I}(b)\right)\right\} = L_P(\sigma)\!\left(\left\{\phi^I_{0_I}(a),\phi^I_{0_I}(b)\right\}\right).\]
But, by the definition, we have that the right hand side is
\begin{align*}
    L_P(\sigma)\!\left(\left\{\phi^I_{0_I}(a),\phi^I_{0_I}(b)\right\}\right) = L_P(\sigma)\!\left(\phi^I_{0_I}(c_{a,b})\right) = \phi^J_{0_J}\!\left(\beta(c_{a,b})\right),
\end{align*}
and the left hand side is
\begin{align*}
    \left\{L_P(\sigma)\!\left(\phi^I_{0_I}(a)\right),L_P(\sigma)\!\left(\phi^I_{0_I}(b)\right)\right\} = \left\{\phi^J_{0_J}(\beta(a)),\phi^J_{0_J}(\beta(b))\right\} = \phi^J_{0_J} \!\left(c_{\beta(a),\beta(b)}\right).
\end{align*}
Moreover, comparing these two values, we find
\begin{align*}
    &\left\Vert\phi^J_{0_J}\!\left(\beta(c_{a,b})\right) - \phi^J_{0_J}\!\left(c_{\beta(a),\beta(b)}\right)\right\Vert_{0_J} \\
    &\quad= \lim_{\hbar\to 0_I}\left\Vert\phi^J_{\alpha(\hbar)}\!\left(\beta(c_{a,b})\right) - \phi^J_{\alpha(\hbar)}\!\left(c_{\beta(a),\beta(b)}\right)\right\Vert_{\alpha(\hbar)}\\
    &\quad= \lim_{\hbar\to 0_I} \left\Vert\frac{i}{\abs{\hbar}_I}\phi^J_{\alpha(\hbar)}\!\left(\beta([a,b])\right) - \frac{i}{\abs{\alpha(\hbar)}_J}\left[\phi^J_{\alpha(\hbar)}(\beta(a)),\phi^J_{\alpha(\hbar)}(\beta(b))\right]\right\Vert_{\alpha(\hbar)}\\
    &\quad= \lim_{\hbar\to 0_I}\left\Vert\frac{i}{\abs{\hbar}_I}\phi^J_{\alpha(\hbar)}(\beta([a,b])) - \frac{i}{\abs{\alpha(\hbar)}_J}\phi^J_{\alpha(\hbar)}(\beta([a,b]))\right\Vert_{\alpha(\hbar)}\\
    &\quad= \lim_{\hbar\to 0_I}\left\vert\frac{1}{\abs{\hbar}_I} - \frac{1}{\abs{\alpha(\hbar)}_J}\right\vert \left\Vert\phi^J_{\alpha(\hbar)}(\beta([a,b]))\right\Vert_{\alpha(\hbar)}
    = 0.
\end{align*}
\end{proof}

\noindent This shows that the construction of the classical Poisson bracket from the Lie bracket commutator is also, in a sense, natural.

\section{Conclusion}

In this paper, we have shown that the bundles of C*-algebras used to represent the limits of physical theories dependent on a parameter can be uniquely extended to limiting values of that parameter.  Moreover, by defining morphisms between bundles of C*-algebras, we have shown the functoriality of the determination of limiting C*-algebraic structure, dynamical structure, and Lie brackets in the case of the classical limit.

We remark on two avenues for further investigation.  First, while our results on limiting dynamics apply to a class of Hamiltonians, it is not clear whether they apply to all physically interesting dynamics.  It might be of particular interest to analyze whether or when chaotic dynamics fall under the purview of our results in \S\ref{sec:dyn}.  Second, while our results show how to take the limits of morphisms between bundles of C*-algebras, we dealt only with morphisms that preserve at least as much structure as *-homomorphisms.  It might be of interest to investigate whether one can also determine limits for weaker morphisms, such as the Morita equivalences that \cite{La02, La02a} uses in his investigation of functorial quantization.

\section*{Acknowledgements}

We have removed acknowledgements for the sake of anonymity.


\appendix

\section{From Quantization to Bundles}
\label{app:quanttobun}

We remarked in \S\ref{sec:uniform} that strict deformation quantizations determine continuous bundles of C*-algebras, and that equivalent quantizations determine the \emph{same} continuous bundles.  We now analyze the sense in which this holds using category-theoretic tools.

First, we outline the construction of a $C_0$ bundle of C*-algebras from a strict deformation quantization $((\mathfrak{A}_\hbar,\mathcal{Q}_\hbar)_{\hbar\in I},\mathcal{P})$.  Here, we stick with $C_0$ bundles of C*-algebras to align our discussion with the literature, but the results of \ref{app:equivdefs} show that one could pick any of the equivalent definitions.  The idea is to define a (unique) C*-algebra of vanishingly continuous sections, which we will denote $C^*(\mathcal{Q})\subseteq \prod_{\hbar\in I}\mathfrak{A}_\hbar$, satisfying the condition that for each $A\in\mathcal{P}$, there is an $a\in C^*(\mathcal{Q})$ such that $a(\hbar) = \mathcal{Q}_\hbar(A)$.  To that end, we define the smallest *-subalgebra $\mathcal{Q}\subseteq\prod_{\hbar\in I}\mathfrak{A}_\hbar$ containing each such $a = [\hbar\mapsto \mathcal{Q}_\hbar(A)]$ as follows
\begin{align*}
   & \mathcal{Q} = \text{span}_\mathbb{C}\left(\left\{ a\in \prod_{\hbar\in I}\mathfrak{A}_\hbar\ \middle|\ a = [\hbar\mapsto \mathcal{Q}_\hbar(A_1)\cdot...\cdot\mathcal{Q}_\hbar(A_n)]\text{ for some $A_1,...,A_n\in\mathcal{P}$}\right\}\right).
\end{align*}
(Here, we explicitly reference the scalars $\mathbb{C}$ in the linear span, denoted $\text{span}_\mathbb{C}$, to distinguish from what follows.) It should be clear that for this construction to work, we need to restrict attention to strict quantizations satisfying the condition that $\hbar\mapsto \norm{\mathcal{Q}_\hbar(F)}_\hbar$ is continuous for every polynomial $F$ of classical quantities $A_1,...,A_n\in\mathcal{P}$.  Further, to ensure the collection of vanishingly continuous sections $C^*(\mathcal{Q})$ satisfies the vanishing completeness condition, we must also include each section that differs in norm from all elements of $\mathcal{Q}$ by a continuous function vanishing at infinity, as follows:
\begin{align*}
    C^*(\mathcal{Q}) = \left\{a'\in\prod_{\hbar\in I}\mathfrak{A}_\hbar\ \middle |\ [\hbar\mapsto\norm{a'(\hbar) - a(\hbar)}_\hbar]\in C_0(I)\text{ for each $a\in\text{span}_\mathbb{C} \mathcal{Q}$}\right\}
\end{align*}
\cite[Theorem 1.2.3-4, p. 111]{La98b} shows that this construction produces a unique $C_0$ bundle of C*-algebras $((\mathfrak{A}_\hbar, \phi_\hbar^I)_{\hbar\in I}, \mathfrak{A}^0)$ by defining
\begin{align*}
    &\mathfrak{A}^0 = C^*(\mathcal{Q})\\
    &\phi_\hbar^I(a) = a(\hbar)\ \text{ for each $a\in C^*(\mathcal{Q})$}.
\end{align*}

We would like to understand the construction of a bundle of C*-algebras from a strict deformation quantization as functorial.  To that end, we define a category of strict deformation quantizations that encodes the structure needed for the construction.

\begin{definition}
The category $\mathbf{SDQuant}$ consists in:
\begin{itemize}
    \item \textit{objects}: strict deformation quantizations $\mathcal{A}_I^D= ((\mathfrak{A}_\hbar,\mathcal{Q}^A_\hbar)_{\hbar\in I}, \mathcal{P}_A)$, where $I\subseteq\mathbb{R}$ is locally compact and contains $0$, and $\hbar\mapsto \norm{\mathcal{Q}_\hbar(F)}_\hbar$ is continuous for every polynomial $F$ of classical quantities $A_1,...,A_n\in\mathcal{P}$.
    \item \textit{morphisms}: a morphism $\sigma : \mathcal{A}_I^D\to\mathcal{B}_J^D$ between strict deformation quantizations $\mathcal{A}_I^D = ((\mathfrak{A}_\hbar,\mathcal{Q}_\hbar^A)_{\hbar\in I}, \mathcal{P}_A)$ and $\mathcal{B}_J^D = ((\mathfrak{B}_\hbar,\mathcal{Q}_\hbar^B)_{\hbar\in I}, \mathcal{P}_B)$ consists in a triple
    \begin{align*}
        \sigma = \left (\alpha,\left (\sigma_\hbar \right)_{\hbar\in I}, \rho \right),
    \end{align*}
where $\alpha: I\to J$ is a metric map such that $\alpha(0) = 0$, the map $\rho: C_0(I)\to C_0(J)$ is a *-homomorphism, and for each $\hbar\in I$, $\sigma_\hbar: \mathfrak{A}_\hbar\to\mathfrak{B}_{\alpha(\hbar)}$ is a *-homomorphism such that $\sigma_0[\mathcal{P}_A]\subseteq\mathcal{P}_B$ and $\sigma_{0|\mathcal{P}_A}$ is a Poisson morphism.  Moreover, we require that these maps satisfy
\begin{align*}
    \mathcal{Q}_{\alpha(\hbar)}^B(\sigma_0(A)) &= \sigma_\hbar \left(\mathcal{Q}_\hbar^A(A)\right)\\
    \rho(f)(\alpha(\hbar)) &= f(\hbar).
\end{align*}
for each $A\in \mathcal{P}_A$ and $f\in C_0(I)$.
\end{itemize}
\end{definition}

We should pause to make one remark concerning the additional map $\rho: C_0(I)\to C_0(J)$ in our definition of a morphism.  To motivate the inclusion of $\rho$ as part of a structure-preserving map, we note the importance of the structure $C_0(I)$ in this construction of a $C_0$ bundle.  As \cite{La02a} remarks, the C*-algebra of sections of a $C_0$ bundle of C*-algebras can be thought of as a $C_0(I)$-algebra in a canonical way, and this structure is sufficient to determine the bundle (see \cite[p. 737-8]{La17} and \cite{Ni96}).  Along with this, one can characterize the sections determined by a quantization map directly in terms of the action of $C_0(I)$.

\begin{proposition}
$C^*(\mathcal{Q}) = \overline{\mathrm{span}_{C_0(I)}(\mathcal{Q})}$, where the overline denotes the closure under the supremum norm in $\prod_{\hbar\in I}\mathfrak{A}_\hbar$.
\end{proposition}

\begin{proof}
We prove equality by proving ($\subseteq$) and ($\supseteq$).

($\subseteq$) Let $a\in C^*(\mathcal{Q})$.  The argument in Lemma 1.2.2 of \cite[p. 110]{La17} implies $a\in \overline{\mathrm{span}_{C_0(I)}(\mathcal{Q})}$.

($\supseteq$) Let $a\in\overline{\mathrm{span}_{C_0(I)}(\mathcal{Q})}$.  Then there is a sequence $a_n\in \mathrm{span}_{C_0(I)}(\mathcal{Q})$ such that $a_n\to a$ in norm as $n\to \infty$.  But each $a_n\in C^*(\mathcal{Q})$ since $C^*(\mathcal{Q})$ forms the algebra of vanishingly continuous sections of a bundle. Hence, since $C^*(\mathcal{Q})$ is norm closed, $a\in C^*(\mathcal{Q})$.
\end{proof}
\noindent Hence, there is reason to think of the action of $C_0(I)$ on continuous sections as an essential part of the structure of a strict deformation quantization, at least as it is used to construct a bundle of C*-algebras.  We encode this action in the structure of our category by requiring structure-preserving morphisms of quantizations to preserve the action of $C_0(I)$.

Now we define a functor $H: \mathbf{SDQuant}\to \mathbf{VBunC^*Alg}$ that acts on objects and moprhism by
\begin{align*}
    \left(\left(\mathfrak{A}_\hbar,\mathcal{Q}_\hbar^A\right)_{\hbar\in I},\mathcal{P}_A\right)\, &\mapsto \, \left(\left(\mathfrak{A}_\hbar,\phi^I_\hbar\right),C^*\hspace{-.3em}\left(\mathcal{Q^A}\right)\right) \\
    \left(\alpha,\left(\sigma_\hbar\right)_{\hbar\in I},\rho\right) \, &\mapsto \, (\alpha,\beta),
\end{align*}
where $\beta: C^*(\mathcal{Q}^A)\to C^*(\mathcal{Q}^B)$ is the unique linear continuous extension of
\begin{align*}
    \beta \left( \left[\hbar\mapsto f(\hbar)\mathcal{Q}^A_\hbar(A_1)\cdot...\cdot\mathcal{Q}^A_\hbar(A_n)\right] \right) = \left [\hbar\mapsto \rho(f)(\hbar)\mathcal{Q}^B_\hbar(\sigma_0(A_1))\cdot...\cdot\mathcal{Q}_\hbar^B(\sigma_0(A_n)) \right]
\end{align*}
for $f\in C_0(I)$ and $A_1,...,A_n\in\mathcal{P}_A$. Indeed, it follows from the definitions that $(\alpha,\beta)$ is a morphism in $\mathbf{VBunC^*Alg}$ because for each $\hbar\in I$ and section $a\in C^*(\mathcal{Q}^A)$, we know $\sigma_\hbar(\phi^I_\hbar(a)) = \phi_{\alpha(\hbar)}^J(\beta(a))$ and $\sigma_\hbar$ is well-defined.  Hence, with this definition, $H$ is a functor.

\section{Continuity and Uniform Continuity}
\label{app:equivdefs}

Each of our two definitions of bundles of C*-algebras in \S\ref{sec:quant} and \S\ref{sec:uniform} comes with an associated category of models.  We can use these categories $\mathbf{VBunC^*Alg}$ and $\mathbf{UBunC^*Alg}$ to analyze the relationship between these definitions.  We will prove the following result:
\begin{proposition}
\label{prop:UVequiv}
There is a categorical equivalence
\[F_{UV}: \mathbf{UBunC^*Alg}\leftrightarrows \mathbf{VBunC^*Alg}: G_{VU}.\]
\end{proposition}

\noindent 
\noindent This shows the intended sense in which the definition of $C_0$ bundles of C*-algebras (Definition \ref{def:vbun}, which appears in the literature) is equivalent to our definition of $\mathit{UC}_b$ bundles of C*-algebras (Definition \ref{def:ubun}).  Although the models satisfying each definition are distinct, there is a correspondence between the structure-preserving maps in the relevant categories.  In this sense, our definition captures the same mathematical structure as the existing definition in the literature.

Roughly, the construction is as follows. $F_{UV}$ takes each object and morphism in the category $\mathbf{UBunC^*Alg}$ and restricts it to the appropriate subcollection of sections whose norm vanishes at infinity.  For the reverse construction, we note that a collection of vanishingly continuous sections is dense in a collection of uniformly continuous sections in the locally uniform topology.  $G_{VU}$ extends each object in $\mathbf{VBunC^*Alg}$ by taking the completion of the collection of vanishingly continuous sections in this topology and then restricting to the subcollection of sections whose norm are uniformly continuous and bounded.  Similarly, $G_{VU}$ extends morphisms continuously to the completion of the collection of sections and then restricts them.  Proposition \ref{prop:UVequiv} establishes that these constructions are natural and preserve structure.   This justifies us in working exclusively with $\mathit{UC}_b$ bundles of C*-algebras, as we do throughout both this paper and the sequel.

We now proceed to the proof establishing a categorical equivalence between $\mathbf{VBunC^*Alg}$ and $\mathbf{UBunC^*Alg}$.  Recall that a functor $F:\mathbf{C}\to \mathbf{D}$ is a categorical equivalence iff there is a functor $G:\mathbf{D}\to\mathbf{C}$ that is ``almost inverse" to $F$ in the sense that for each object $X$ in $\mathbf{C}$, there is an isomorphism $\eta_X$ from $G\circ F(X)$ to $X$ in $\mathbf{C}$ that makes the following diagram commute for all morhpisms $f:X\to Y$ in $\mathbf{C}$:
\begin{equation*}
\begin{tikzcd}
G\circ F(X) \arrow{r}{\eta_X} \arrow{d}[swap]{G\circ F(f)}
& X\arrow{d}{f} \\
G\circ F(Y) \arrow{r}{\eta_Y}
& Y
\end{tikzcd}
\end{equation*}
(And vice versa for $F\circ G$ and objects in $\mathbf{D}$.) In this case, $\eta$ is a \emph{natural isomorphism} between $G\circ F$ and the identity functor $1_{\mathbf{C}}$.  One can characterize or establish a categorical equivalence by providing this ``almost inverse'' functor \cite[Proposition 7.26, p. 173]{Aw10}.

In the case of $\mathbf{VBunC^*Alg}$ and $\mathbf{UBunC^*Alg}$, we will provide a pair of almost inverse functors through an intermediary category.  This intermediary category corresponds to a further definition of a continuous bundle-like structure of C*-algebras that is often used in the mathematical literature, which we prove to be equivalent along the way.

\begin{definition}[\cite{Di77}]
\label{def:contfield}
A \emph{continuous field of C*-algebras} over a (now arbitrary) topological space $I$ is a family of C*-algebras $(\mathfrak{A}_\hbar)_{\hbar\in I}$ and a subset $\Gamma\subseteq \prod_{\hbar\in I}\mathfrak{A}_\hbar$.  In other words, each element $a\in \Gamma$ is a map $a:I\mapsto \coprod_{\hbar\in I}\mathfrak{A}_\hbar$ such that $a(\hbar)\in\mathfrak{A}_\hbar$.  We call $\Gamma$ the collection of \emph{continuous sections}, we require it to be a *-algebra under pointwise operations, and we require it to satisfy the following conditions:
\begin{enumerate}[(i)]
\item \emph{Density.} For each $\hbar\in I$, the set $\{a(\hbar)\ |\ a\in \Gamma\}$ is dense in $\mathfrak{A}_\hbar$.
\item \emph{Locally uniform closure.} If $a\in \prod_{\hbar\in I}\mathfrak{A}_\hbar$ and for each $\hbar\in I$ and $\epsilon >0$, there is some $a'\in \Gamma$ such that $\sup_{\hbar'\in U}\norm{a(\hbar') - a'(\hbar')}_{\hbar'}\leq \epsilon$ for some neighborhood $U$ of $\hbar$, then $a\in \Gamma$.
\item \emph{Continuity.} For each $a\in \Gamma$, the map $N_a: \hbar\mapsto \norm{a(\hbar)}_\hbar$ is continuous.
\end{enumerate}
\end{definition}

In Definition \ref{def:contfield}, note that although the collection $\Gamma$ of continuous sections is a *-algebra, it is not in general a C*-algebra in a natural way.  In general, elements of $\Gamma$ may be unbounded in the supremum norm.  One can see this, for example, because the condition of locally uniform closure implies
\begin{enumerate}[(ii*)]
\item \emph{Completeness.} For any continuous $f:I\to\mathbb{C}$ and $a\in\Gamma$, there is an element $fa\in\Gamma$ such that $fa(\hbar) = f(\hbar)a(\hbar)$.
\end{enumerate}
Since there may be unbounded continuous functions on $I$, this implies that there may be unbounded continuous sections.  However, it is known that one can associate with each continuous field of C*-algebras $((\mathfrak{A}_\hbar)_{\hbar\in I},\Gamma)$ a canonical C*-algebra
\[\Gamma^0_\mathfrak{A} := \left\{a\in \Gamma\ \middle |\ N_a\in C_0(I) \right\},\]
which consists in continuous sections whose norm vanishes at infinity, with pointwise operations.  Similarly, one can associate with each continuous field of C*-algebras $((\mathfrak{A}_\hbar)_{\hbar\in I},\Gamma)$ over a metric space $I$ a canonical C*-algebra
\[\Gamma^U_\mathfrak{A} := \left \{a\in\Gamma\ \middle|\ N_a\in \mathit{UC}_b(I)\right\},\]
which consists in continuous sections whose norm is uniformly continuous and bounded, again with pointwise operations.  We now aim to understand the construction of canonical $C_0$ and $\mathit{UC}_b$ bundles associated with $\Gamma^0_\mathfrak{A}$ and $\Gamma^U_\mathfrak{A}$ as a functor from the category of continuous fields of C*-algebras, which we define as follows.

\begin{definition}
The category $\mathbf{FieldC^*Alg}$ consists in
\begin{itemize}
    \item \textit{objects}: continuous fields of C*-algebras whose base space is a locally compact metric space,
    \[\mathcal{A}^F_I = \left ( \left (\mathfrak{A}_\hbar \right)_{\hbar\in I},\Gamma_\mathfrak{A} \right);\]
    \item \textit{morphisms}: homomorphisms $\sigma: \mathcal{A}^F_I\to\mathcal{B}^F_J$, where each $\sigma$ is a pair of maps
    \[\sigma = (\alpha,\beta),\]
    where $\alpha: I\to J$ is a metric map, $\beta: \Gamma_\mathfrak{A}\to\Gamma_\mathfrak{B}$ is a *-homomorphism, and for all $a_1,a_2\in \Gamma_\mathfrak{A}$ and $\hbar\in I$, if
        $a_1(\hbar) = a_2(\hbar)$, then  
        $\beta(a_1)(\alpha(\hbar)) = \beta(a_2)(\alpha(\hbar))$.
\end{itemize}
\end{definition}

We define a pair of functors relating $\mathbf{FieldC^*Alg}$ and $\mathbf{VBunC^*Alg}$ that captures the construction described by \cite[p. 677-8]{KiWa95} and we show that this construction indeed provides a categorical equivalence.  We then use an exactly analogous construction to establish a categorical equivalence between $\mathbf{FieldC^*Alg}$ and $\mathbf{UBunC^*Alg}$.  First, note that there is a difference between continuous fields of C*-algebras and (vanishingly or uniformly) continuous bundles of C*-algebras that is somewhat trivial: in a continuous field, the sections are directly specified as functions from the base up to the fibers, whereas in a bundle, the sections are given by an ``independent" C*-algebra that is connected to the fibers via the evaluation maps.  To aid in the translation, if we are given a $C_0$ bundle of C*-algebras $((\mathfrak{A}_\hbar,\phi_\hbar^I)_{\hbar\in I},\mathfrak{A}^0)$, we define
\[\hat{\mathfrak{A}}^0 := \left\{\gamma_a\in \prod_{\hbar\in I}\mathfrak{A}_\hbar\ \middle |\ \gamma_a(\hbar) = \phi_\hbar(a) \text{ for some } a\in\mathfrak{A}^0 \right\}.\]
Similarly, if we are given a $\mathit{UC}_b$ bundle of C*-algebras $((\mathfrak{A}_\hbar,\phi_\hbar^I)_{\hbar\in I},\mathfrak{A})$, we define
\[\hat{\mathfrak{A}} := \left \{\gamma_a\in \prod_{\hbar\in I}\mathfrak{A}_\hbar\ \middle |\ \gamma_a(\hbar) = \phi_\hbar(a) \text{ for some } a\in\mathfrak{A}\right\}.\]

Now we can define a functor $F_{FV}: \mathbf{FieldC^*Alg}\to\mathbf{VBunC^*Alg}$ on objects and morphisms by
\begin{align*}
    \mathcal{A}^F_I = \left ( \left(\mathfrak{A}_\hbar \right)_{\hbar\in I},\Gamma_{\mathfrak{A}}\right) \, &\mapsto \, \mathcal{A}^V_I := \left( \left(\mathfrak{A}_\hbar,\phi^I_\hbar\right)_{\hbar\in I},\Gamma_\mathfrak{A}^0\right) \\
    (\alpha,\beta) \, &\mapsto \, \left(\alpha,\beta_{|\Gamma^0_\mathfrak{A}}\right),
\end{align*}
where $\phi_\hbar^I (a) := a(\hbar)$ for all $a\in\Gamma_\mathfrak{A}^0$ and $\hbar\in I$. One can easily check this induces a map from $\mathrm{Hom}_\mathbf{FieldC^*Alg}(\mathcal{A}_I^F,\mathcal{B}_J^F)$ to $\mathrm{Hom}_\mathbf{VBunC^*Alg}(\mathcal{A}_I^V,\mathcal{B}_J^V)$, i.e., $\beta(a)$ has vanishingly continuous norm for any $a$ with vanishingly continuous norm.  This follows from Lemma \ref{lem:fibmorph}.

Further, define a functor $G_{VF}: \mathbf{VBunC^*Alg}\to\mathbf{FieldC^*Alg}$ on objects and morphisms by
\begin{align*}
    \mathcal{A}_I^V = \left(\left(\mathfrak{A}_\hbar,\phi_\hbar^I \right)_{\hbar\in I},\mathfrak{A}^0\right) \, &\mapsto \, \mathcal{A}_I^F:= \left(\left(\mathfrak{A}_\hbar\right)_{\hbar\in I},\mathfrak{A}_\Gamma\right) \\
    (\alpha,\beta) \, &\mapsto \, \left(\alpha,\beta_\Gamma \right),
\end{align*}
where $\mathfrak{A}_\Gamma$ is defined as the completion of $\hat{\mathfrak{A}}^0$ in the topology $\tau_{\mathit{lu}}$ of locally uniform convergence, i.e.,
\begin{equation*}
    \mathfrak{A}_\Gamma:= \hspace*{-.5 em}\oset[-.16em]{\hspace*{.75em}\line(1,0){14}_{\tau_{\mathit{lu}}}}{\hat{\mathfrak{A}}^0}\hspace*{-.3em}.
\end{equation*}
More explicitly, $\tau_{\mathit{lu}}$ is the vector space topology that has a neighborhood base at $0$ consisting in sets of the form 
\[O(K,\epsilon):=\left\{\gamma_a\in \hat{\mathfrak{A}}^0\ \middle|\ \sup_{\hbar\in K}\left \Vert\gamma_a(\hbar)\right\Vert_\hbar <\epsilon \right\}\]
for each $\epsilon>0$ and compact $K\subseteq I$.  Elements of $\mathfrak{A}_\Gamma$ are equivalence classes of $\tau_{\mathit{lu}}$-Cauchy nets with pointwise operations and limiting norm \cite[Ex. 4.1]{BaFrInTr06}. Similarly,
$$
\beta_\Gamma: = \hspace*{-.5em}\oset[-.16em]{\hspace*{1em}\line(1,0){7}_{\tau_{\mathit{lu}}}}{\hat{\beta}},
$$
where $\hat{\beta}(\gamma_a) = \gamma_{\beta(a)}$ for each $a\in\mathfrak{A}^0$, and the overline denotes the unique linear, continuous (relative to $\tau_{lu}$) extension of the map $\hat{\beta}$ to the completion $\mathfrak{A}_\Gamma$.  We have the following result.

\begin{lemma}
\label{lem:FVequiv}
$F_{FV}:\mathbf{FieldC^*Alg}\leftrightarrows \mathbf{VBunC^*Alg}:G_{VF}$ is a categorical equivalence.
\end{lemma}

\begin{proof}
We specify directly two natural isomorphisms $\chi: G_{VF}\circ F_{FV}\to 1_F$, where and $1_F$ is the identity functor on $\mathbf{FieldC^*Alg}$, and $\eta: F_{FV}\circ G_{VF}\to 1_V$, where $1_V$ is the identity functor on $\mathbf{VBunC^*Alg}$.

First, for each object $\mathcal{A}_I^F = ((\mathfrak{A}_\hbar)_{\hbar\in I},\Gamma_\mathfrak{A})$ in $\mathbf{FieldC^*Alg}$, we define
$$
\chi_\mathcal{A}: G_{VF}\circ F_{FV} \left (\mathcal{A}_I^F \right)\to \mathcal{A}_I^F
$$
by $\chi_\mathcal{A} := (\mathrm{id}_I,\beta_\mathcal{A})$, where $\mathrm{id}_I$ is the identity map on $I$ so it only remains to specify the map
$$
\beta_\mathcal{A}:  \hspace*{-.4em}\oset[-.16em]{\hspace*{.75em}\line(1,0){14}_{\tau_{\mathit{lu}}}}{\hat{\Gamma}_\mathfrak{A}^0}\to \Gamma_\mathfrak{A}.
$$
We define $\beta_\mathcal{A}$ for each $[a_\delta]\in   \hspace*{-.4em}\oset[-.16em]{\hspace*{.75em}\line(1,0){14}_{\tau_{\mathit{lu}}}}{\hat{\Gamma}_\mathfrak{A}^0}$ by
\[\beta_\mathcal{A}([a_\delta]) = \overset{\tau_{\mathit{lu}}}{\lim_\delta}\ a_\delta ,\]
where $[a_\delta]$ is the equivalence class of nets $b_\delta$ such that $(a_\delta - b_\delta)\to 0$ in $\tau_{\mathit{lu}}$.  We need to show that $\beta_\mathcal{A}$ is a *-isomorphism.  First, note that $\beta_\mathcal{A}$ is well-defined because if $[a_\delta] = [b_\delta]$, then the limits $\lim_{\delta} a_\delta = \lim_\delta b_\delta$ exist and are identical (where these limits are taken relative to the topology $\tau_{\mathit{lu}}$).  Similarly, $\beta_\mathcal{A}$ is injective because if $\lim_\delta a_\delta = \lim_\delta b_\delta$, then $[a_\delta] = [b_\delta]$.

Further, $\beta_\mathcal{A}$ is surjective: consider any $a\in\Gamma_\mathfrak{A}$.  We will construct a net in $\hat{\Gamma}_\mathfrak{A}^0$ converging to $a$ with respect to $\tau_{\mathit{lu}}$.  Our net will be indexed by the directed set of pairs $(K,\epsilon)$, where $K$ is a compact subset of $I$ and $\epsilon>0$ with $(K,\epsilon)\preceq (K',\epsilon')$ iff $K\subseteq K'$ and $\epsilon'\leq \epsilon$.  For each compact subset $K\subseteq I$ and $\epsilon >0$, we consider the function $N_{a|K}: \hbar\in K\mapsto \norm{a(\hbar)}_\hbar$.  The Tietze extension theorem \cite[Theorem 35.1, p. 219]{Mu00} implies there is a $f^K\in C_0(I)$ such that $f^K_{|K} = N_{a|K}$.  Now define $g_{K,\epsilon}: I\to \mathbb{R}$ by
\[g_{K,\epsilon}(\hbar):= \frac{f^K(\hbar)}{N_a(\hbar) + \epsilon}\]
for all $\hbar\in I$.  Since $g_{K,\epsilon}$ is continuous, $a_{K,\epsilon}:= g_{K,\epsilon}a\in \Gamma_\mathfrak{A}$.  Then for each $\hbar\in I$,
\[\norm{a_{K,\epsilon}(\hbar)}_\hbar = \abs{g_{K,\epsilon}(\hbar)}\cdot\norm{a(\hbar)}_\hbar = \abs{f^K(\hbar)}\cdot \frac{N_a(\hbar)}{N_a(\hbar) + \epsilon}\leq \abs{f^K(\hbar)},\]
and hence, $N_{a_{K,\epsilon}}\in C_0(I)$.  Further,
\begin{align*}
\sup_{\hbar\in K}\ \norm{(a-a_{K,\epsilon})(\hbar)}_\hbar &= \sup_{\hbar\in K}\ \norm{(1-g_{K,\epsilon})(a(\hbar))}_\hbar\\
&= \sup_{\hbar\in K}\ \abs{1-g_{K,\epsilon}(\hbar)}\cdot N_a(\hbar) \leq \epsilon\cdot  \frac{N_a(\hbar)}{N_a(\hbar)+\epsilon} < \epsilon.
\end{align*}
Hence, for each compact $K\subseteq I$ and each $\epsilon >0$, for all pairs $(K',\epsilon')\succeq (K,\epsilon)$, we have
\[\sup_{\hbar\in K}\ \norm{(a - a_{K',\epsilon'})(\hbar)}_\hbar\leq \sup_{\hbar\in K'}\ \norm{(a-a_{K',\epsilon'})(\hbar)}_\hbar < \epsilon' \leq \epsilon,\]
which implies $\beta_\mathcal{A}([a_{K,\epsilon}]) = a$.  Hence, $\beta_\mathcal{A}$ is surjective.

Finally, $\beta_\mathcal{A}$ is linear, multiplicative, and *-preserving because operations are defined pointwise on the completion of $\hat{\Gamma}_\mathfrak{A}^0$ with respect to $\tau_{\mathit{lu}}$. Hence, we have established that $\chi_\mathcal{A} = (\mathrm{id}_I,\beta_\mathcal{A})$ is an isomorphism in $\mathbf{FieldC^*Alg}$.

Moreover, clearly the following diagram commutes for each arrow $\sigma: \mathcal{A}_I^F\to\mathcal{B}_J^F$ in $\mathbf{FieldC^*Alg}$:
\begin{equation*}
\begin{tikzcd}
G_{VF}\circ F_{FV} \left(\mathcal{A}_I^F \right) \arrow{rr}{\chi_\mathcal{A}} \arrow{dd}[swap]{G_{VF}\circ F_{FV}(\sigma)}
& & 1_F\left(\mathcal{A}_I^F \right)\arrow{dd}{1_F(\sigma)} \\
& & \\
G_{VF}\circ F_{FV}\left(\mathcal{B}_J^F\right) \arrow{rr}{\chi_\mathcal{B}}
& & 1_F\left(\mathcal{B}_J^F\right)
\end{tikzcd}
\end{equation*}
\noindent
\noindent
Hence, $\chi$ is a natural isomorphism.

Conversely, for each object $\mathcal{A}_I^V = ((\mathfrak{A}_\hbar,\phi^I_\hbar)_{\hbar\in I},\mathfrak{A}^0)$ in $\mathbf{VBunC^*Alg}$, we define $\eta_\mathcal{A}: F_{FV}\circ G_{VF}(\mathcal{A}_I^V)\to 1_V(\mathcal{A}_I^V)$ by $\eta_\mathcal{A} = (\mathrm{id}_I,\beta_\mathcal{A})$, so that it only remains to specify the map $\beta_\mathcal{A}: (\mathfrak{A}_\Gamma)^0\to\mathfrak{A}^0$.  We define $\beta_\mathcal{A}$ for each $[a_\delta]\in (\mathfrak{A}_\Gamma)^0$ by
\[\beta_\mathcal{A}([a_\delta]) = \overset{\tau_{\mathit{lu}}}{\lim_\delta}\ a_\delta.\]
Note that $\beta_\mathcal{A}$ is well-defined because if $[a_\delta] = [b_\delta]$, then $\lim_\delta a_\delta = \lim_\delta b_\delta$.  Further, since $N_{[a_\delta]}\in C_0(I)$, it follows that for $a = \lim_\delta a_\delta$, $N_a\in C_0(I)$.  Of course, $\beta_\mathcal{A}$ is surjective because for each $a\in \mathfrak{A}^0$, the net $a_\delta = a$ for all $\delta$ is such that $\beta_\mathcal{A}([a_\delta]) = a$.  Further, $\beta_\mathcal{A}$ is injective because $\lim_\delta a_\delta = \lim_\delta b_\delta$ implies $[a_\delta] = [b_\delta]$.  Finally, $\beta_\mathcal{A}$ is linear, multiplicative, and *-preserving because addition and involution are $\tau_{\mathit{lu}}$-continuous and multiplication is jointly $\tau_{\mathit{lu}}$-continuous.  Hence, we have established that $(\mathrm{id}_I,\beta_\mathcal{A})$ is an isomorphism in $\mathbf{VBunC^*Alg}$.  Clearly, the following diagram commutes for each arrow $\sigma: \mathcal{A}_I^V\to\mathcal{B}_J^V$ in $\mathbf{VBunC^*Alg}$:
\begin{equation*}
\begin{tikzcd}
F_{FV}\circ G_{VF}\left ( \mathcal{A}_I^V \right ) \arrow{rr}{\eta_\mathcal{A}} \arrow{dd}[swap]{G_{VF}\circ F_{FV}(\sigma)}
& & 1_V \left (\mathcal{A}_I^V \right)\arrow{dd}{1_V(\sigma)} \\
& & \\
F_{FV}\circ G_{VF}\left ( \mathcal{B}_J^V \right) \arrow{rr}{\eta_\mathcal{B}}
& & 1_V \left(\mathcal{B}_J^V\right)
\end{tikzcd}
\end{equation*}
\noindent Hence, $\eta$ is a natural isomorphism.
\end{proof}

Next, we define a pair of functors relating $\mathbf{FieldC^*Alg}$ and $\mathbf{UBunC^*Alg}$ that provide a categorical equivalence.  We define $F_{FU}: \mathbf{FieldC^*Alg}\to \mathbf{UBunC^*Alg}$ on objects and morphisms by
\begin{align*}
    \mathcal{A}^F_I = \left(\left(\mathfrak{A}_\hbar\right)_{\hbar\in I},\Gamma_{\mathfrak{A}}\right) \, &\mapsto \, \mathcal{A}^U_I := \left(\left(\mathfrak{A}_\hbar,\phi^I_\hbar\right)_{\hbar\in I},\Gamma_\mathfrak{A}^U\right)\\
   (\alpha,\beta)\, &\mapsto \, \left (\alpha,\beta_{|\Gamma^U_\mathfrak{A}} \right ),
\end{align*}
where $\phi_\hbar (a) := a(\hbar)$ for all $a\in\mathfrak{A}$ and $\hbar\in I$. As previously, it follows from Lemma \ref{lem:fibmorph} below that this induces a map from $\mathrm{Hom}_\mathbf{FieldC^*Alg}(\mathcal{A}_I^F,\mathcal{B}_J^F)$ to $\mathrm{Hom}_\mathbf{UBunC^*Alg}(\mathcal{A}_I^U,\mathcal{B}_J^U)$.

Further, define a functor $G_{UF}: \mathbf{UBunC^*Alg}\to\mathbf{FieldC^*Alg}$ on objects and morphisms by
\begin{align*}
    \mathcal{A}_I^U = \left(\left(\mathfrak{A}_\hbar,\phi_\hbar^I\right)_{\hbar\in I},\mathfrak{A}\right) \, &\mapsto \, \mathcal{A}_I^F := \left(\left(\mathfrak{A}_\hbar\right)_{\hbar\in I},\hspace*{-.5em}\oset[-.16em]{\hspace*{.8em}\line(1,0){9}_{\tau_{\mathit{lu}}}}{\hat{\mathfrak{A}}}\hspace*{-.05em}\right)\\
    (\alpha,\beta) \, &\mapsto \, \left(\alpha, \hspace*{-.75em}\oset[-.16em]{\hspace*{1em}\line(1,0){7}_{\tau_{\mathit{lu}}}}{\hat{\beta}} \right),
\end{align*}
where $\hat{\beta}(\gamma_a) = \gamma_{\beta(a)}$ for each $a\in\mathfrak{A}$, and the overline denotes the unique linear, continuous extension of the map $\beta$ to the $\tau_{\mathit{lu}}$-completion of $\hat{\mathfrak{A}}$. Our next categorical equivalence follows then from the same argument used for Lemma \ref{lem:FVequiv}, given appropriate replacements of $V$ (or $0$) by $U$.

\begin{lemma}
$F_{FU}: \mathbf{FieldC^*Alg}\leftrightarrows\mathbf{UBunC^*Alg}: G_{UF}$ is a categorical equivalence.
\end{lemma}

Now we define $F_{UV} = F_{FV} \circ G_{UF}$ and $G_{VU} = F_{FU}\circ G_{VF}$.  That these functors provide an equivalence follows immediately from the preceding two lemmas.  Hence, this completes the proof of Proposition \ref{prop:UVequiv}.

\bibliographystyle{plain}
\bibliography{bibliography.bib}

\end{document}